\documentclass[11pt]{article}

\usepackage{epsf}

\usepackage{amsmath, amssymb, amsfonts}

\usepackage{enumerate}

\newtheorem{theorem}{\bf Theorem}[section]
\newtheorem{corollary}[theorem]{\bf Corollary}
\newtheorem{lemma}[theorem]{\bf Lemma}

\newtheorem{conjecture}[theorem]{\bf Conjecture}

\newtheorem{fact}[theorem]{\bf Fact}

\newcommand\drop[1]{}

\newcommand{\D}{\ensuremath{\mathcal{D}}}

\newcommand{\V}{\ensuremath{\mathcal{V}}}
\newcommand{\W}{\ensuremath{\mathcal{W}}}

\newcommand{\Y}{\ensuremath{\mathcal{Y}}}

\renewcommand{\P}{\ensuremath{\mathcal{P}}}

\newcommand{\Q}{\ensuremath{\mathcal{Q}}}

\renewcommand{\epsilon}[0]{\varepsilon}
\renewcommand\a{\alpha}
\renewcommand\b{\beta}

\newcommand{\DEF}[1]{{\em #1\/}}
\newcommand\myitemsep{\setlength{\itemsep}{0pt}}

\newcommand{\qed}{\hbox{\rule{6pt}{6pt}}} 
\newcommand{\R}{\mbox{\boldmath $R$}}

\newif\ifplaintheorems
\def\plaintheorems{\plaintheoremstrue}

\let\thmheadfont\bf

\def\showlabel#1{}
\def\showfiglabel#1{}

\renewcommand\int{\mathop{\hbox{{\rm int}}}}

\newcommand{\boundary}{\partial}
\newcommand{\closure}{\overline}

\newcounter{unnumber}

\newcommand\NN{\hbox{{$\mathbb N$}}}

\plaintheorems

\renewcommand\a{\alpha}
\renewcommand\b{\beta}

\let\epsilon=\varepsilon

\makeatletter
\def\GrabProofArgument[#1]{ #1: \egroup\ignorespaces}
\def\proof{\noindent\textbf\bgroup Proof%
           \@ifnextchar[{\GrabProofArgument}{: \egroup\ignorespaces}}

\makeatother

\title{{\bf The odd Hadwiger's conjecture is ``almost'' decidable}}
\author{
Ken-ichi Kawarabayashi\thanks{
National Institute of Informatics, 2-1-2 Hitotsubashi, Chiyoda-ku, Tokyo, Japan.
Research is partly supported by JST ERATO
	Kawarabayashi Large Graph Project and Mitsubishi Foundation.
Email:
{\tt k\_keniti@nii.ac.jp}}
}
\begin{document}

\date{}

\maketitle

\baselineskip 13.3pt

\begin{abstract}
The concept ``odd-minor'' which is a generalization of
minor-relation has received considerable amount of attention by many
researchers, and led to several beautiful conjectures and results.
We say that $H$ has an odd complete minor of order $l$
if there are $l$ vertex disjoint trees in $H$ such that every two of them are
joined by an edge, and in addition, all the vertices of trees are two-colored
in such a way that the edges within the trees are bichromatic, but the edges
between trees are monochromatic. Hence it is easy to see that
odd minor is a generalization of minor.
Let us observe that
the complete bipartite graph $K_{n/2,n/2}$ certainly contains
a $K_k$-minor for $k \leq n/2$, but on the other hand, it does not contain
any odd $K_k$-minor for any $k \geq 3$. So odd-minor-closed graphs seem to be much weaker than minor-closed graphs.

The odd Hadwiger's conjecture, made by Gerads and Seymour in early 1990s, is an analogue of the famous Hadwiger's conjecture.
It says that every graph with no odd $K_t$-minor is $(t-1)$-colorable. This conjecture is known to be true for $t \leq 5$, but
the cases $t \geq 5$ are wide open. So far, the most general result says that every graph with no odd $K_t$-minor is
$O(t \sqrt{\log t})$-colorable.

In this paper, we tackle this conjecture from an algorithmic view, and show the following:

For a given graph $G$ and any fixed $t$, there is a polynomial time algorithm to
output one of the following:

\begin{enumerate}
\item
a $(t-1)$-coloring of $G$, or
\item
an odd $K_{t}$-minor of $G$, or
\item
after making all ``reductions'' to $G$, the resulting graph $H$ (which is an odd minor of $G$ and which has no reductions)
has a tree-decomposition $(T, Y)$ such that torso of each bag $Y_t$ is either
\begin{itemize}
\item
of size at most $f_1(t) \log n$ for some function $f_1$ of $t$, or
\item
a graph that has a vertex $X$ of order at most $f_2(t)$ for some function $f_2$ of $t$ such that $Y_t-X$ is bipartite. Moreover,
degree of $t$ in $T$ is at most $f_3(t)$ for some function $f_3$ of $t$.
\end{itemize}
\end{enumerate}

Let us observe that the last odd minor $H$ is indeed a minimal counterexample to the odd Hadwiger's conjecture for the case $t$.
From this we obtain the following:

\medskip

For a given graph $G$ and any fixed $t$, there is a polynomial time algorithm to
output one of the following:

\begin{enumerate}
\item
a $(t-1)$-coloring of $G$, or
\item
an odd $K_{t}$-minor of $G$, or
\item
after making all ``reductions'' to $G$, we can color the resulting graph $H$ (which is an odd minor of $G$ and which has no reductions) with at most $\chi(H)+1$ colors
in polynomial time, where $\chi(H)$ is the chromatic number of $H$.
\end{enumerate}
In the last conclusion, we can actually figure out whether or not
$H$ contains an odd $K_t$-minor.
\end{abstract}

\vfill \baselineskip 11pt \noindent 17 May 2015, revised \date.

\vfil\eject

\baselineskip 13.0pt

\section{Introduction}

\subsection{Hadwiger's conjecture}
Hadwiger's Conjecture, dated back from 1943,
suggests a far-reaching generalization of the Four Color Theorem
\cite{4ct1,4ct2,4ct3}. Definitely one of the deepest open problems in
graph theory. It says that any graph without $K_k$ as a minor is $(k-1)$-colorable. Let us give some known results for Hadwiger's conjecture
(for more details, see \cite{jt}). In 1937,
Wagner \cite{Wa37} proved that the case $k=5$ of
the conjecture is, indeed, equivalent to the Four Color Theorem.
In 1993, Robertson, Seymour and Thomas \cite{RST1} proved that the case
$k=6$ would follow from the Four Color Theorem. The cases $k \geq 7$ are wide open,
and even for the case $k=7$, the partial result in \cite{kt} has been best known.
It is known that
any graph $G$ without $K_k$ as a minor is $O(k \sqrt{\log k})$-colorable,
see \cite{kost1,thomason1}, but currently, this result is
best known for the general case.

\subsection{The Odd Hadwiger's conjecture}

Recently, the concept ''odd-minor'' has drawn attention by many researchers,
because of its relation to graph minor theory, and Hadwiger's conjecture.
Formally, we say that $H$ has an \DEF{odd complete minor of order $l$}
if there are $l$ vertex disjoint trees in $H$ such that every two of them are
joined by an edge, and in addition, all the vertices of trees can be two-colored
in such a way that the edges within the trees are bichromatic, but the edges
between trees are monochromatic. Hence it is easy to see that
odd minor is a generalization of minor.

Odd-minor-closed graphs seem to be much weaker than minor-closed graphs.
Indeed, the complete bipartite graph $K_{n/2,n/2}$ certainly contains
a $K_k$-minor for $k \leq n/2$, but on the other hand, it does not contain
any odd $K_k$-minor for any $k \geq 3$. In fact,
any graph $G$ without $K_k$-minors is $O(k \sqrt{\log k})$-degenerate, i.e,
every induced subgraph has a vertex of degree at most $O(k \sqrt{\log k})$ (\cite{kost1,thomason1}).
So $G$ has at most $O(k \sqrt{\log k} n)$ edges. On the other hand,
some graphs with no odd-$K_k$-minor may have $\Theta(n^2)$ edges.

This seems to make huge difference. On the other hand,
odd minors are actually motivated by graph minor theory, and many researchers believe that there would be some analogue
of graph minor theory, and some connections to the well-known conjecture
of Hadwiger \cite{hadwiger}.

Concerning the connection to Hadwiger's conjecture, Gerards and Seymour (see \cite{jt}, page 115.) conjectured the
following.
\begin{conjecture}
\label{odd}
For all $l\ge1$, every graph
with no odd  $K_{l+1}$ minor  is  $l$-colorable.
\end{conjecture}
This is an analogue of
Hadwiger's conjecture. In fact, it is easy to see that
Conjecture \ref{odd} immediately
implies Hadwiger's conjecture.
Again, Conjecture \ref{odd} is trivially true when $l=1,\, 2$. In fact,
when $l=2$,
 this means that
if a graph has no odd cycles, then it is $2$-colorable.

The first nontrivial case $l=3$  was proved by Catlin \cite{catlin}. Recently,
Guenin \cite{guenin} announced a solution of the  case $l=4$.
This result would imply  the Four Color Theorem because a graph having an odd $K_5$-minor
certainly contains a $K_5$-minor.
Conjecture \ref{odd} is open for  $l \geq 5$. The only general result is the following.

\begin{theorem}
Any graph with no odd $K_k$-minor is $O(k \sqrt{\log k})$-colorable.
\end{theorem}
This is clearly a generalization of the above mentioned result by Kostochka and Thomason
\cite{kost1,thomason1} for the general case on Hadwiger's conjecture.
This is first proved by Geelen et al. \cite{geelen}.
A simpler and shorter proof is given in \cite{kkk}.

\subsection{Structure theorem}
\label{stodd}
\showlabel{stodd}

As mentioned above, the concept ``odd-minor'' is motivated by the theory of graph minors which  was developed by Robertson and Seymour
in a series of 23 papers published over more than thirty years.
The purpose of the series of papers is to prove the {\em graph minor
theorem\/}, which says that in any infinite collection of finite
graphs there is one that is a minor of another. As with other deep
results in mathematics, the body of theory developed for the proof
has also found applications not only in mathematics and but also in computer science. Yet many of these
applications rely on an auxiliary result which is central to the proof of the
graph minor theorem: a~result which approximately describes the structure of all graphs
$G$ which do not contain some fixed graph~$H$ as a minor, see \cite{RS16}.
At a high level, the theorem says that every such a graph has a tree-decomposition such that
each piece is
\begin{quote}
after deleting bounded number of vertices,
an ``almost'' embedded graph (for precise definition, see later) into a bounded-genus surface.
\end{quote}

Recently, similar structure results are obtained for graphs without some fixed graph~$H$ as an odd-minor \cite{DKM}.
Namely, every graph with no odd $H$-minor has a tree-decomposition such that
each piece is either
\begin{enumerate}
\item
after deleting bounded number of vertices,
an ``almost'' embedded graph into a bounded-genus surface, or
\item
after deleting bounded number of vertices, a bipartite graph.
\end{enumerate}

As we see here, in addition to the minor-free case, for the odd-minor-free case,
we only need to add the second conclusion for the decomposition theorem.
Our purpose in this paper is to use this structure theorem to attack the odd Hadwiger's conjecture.

\subsection{Results}

Let $G$ be a graph satisfying the following conditions:
\begin{enumerate}
\item[\rm (1)]
$G$ is $t$-chromatic.
\item[\rm (2)]
$G$ is minimal with respect to the odd-minor-relation in the class of all
$t$-chromatic graphs.
\item[\rm (3)]
$G$ does not contain $K_t$ as an odd minor.
\end{enumerate}
We call such a graph $G$ a {\it minimal counterexample to the odd Hadwiger's conjecture} for the case $t$.
Conjecture \ref{odd} suggests there are {\it no} such graphs.
Here {\it odd-minor-operation} means the following:
\begin{enumerate}
\item
First, delete edges.
\item
Then take a cut that consists of edges $R$.
Then contract ALL edges in $R$.
\end{enumerate}

It has been known that odd-minors are closed under the odd-minor-relations (see e.g., \cite{ge}).
We say $G$ has a \DEF{reduction} if there is an odd-minor-operation to $G$ such that any $(t-1)$-coloring of the resulting graph
can be extended to a $(t-1)$-coloring of $G$. Let us consider the following case: suppose there is a separation $(A,B)$ of $G$ such that
both $A-B\not=\emptyset$ and $B-A\not=\emptyset$.  Suppose furthermore that there is a $(t-1)$-coloring $\sigma$ of $A$. Then
we may reduce $A$ onto $A \cap B$ via odd-minor-operations so that the resulting graph of $B$ has a $(t-1)$-coloring that extends the coloring of $\sigma$ on $A \cap B$. If this happens, we also say that there is a reduction in $G$.
Hence if $G$ is a minimal counterexample to the odd Hadwiger's conjecture for the case $t$,
then there is no reduction in $G$.

The following result, which is closely related to the minimal counterexample to the odd Hadwiger's conjecture for the case $t$,
has been known from \cite{km}.

\begin{theorem}\label{main4}
For every fixed $k$, there is an algorithm with running time $O(n^3)$
for deciding either that
\begin{enumerate}
\myitemsep
\item[\rm (1)]
a given graph $G$ of order $n$ is $2496k$-colorable, or
\item[\rm (2)]
$G$ contains an odd-$K_{k}$-minor, or
\item[\rm (3)]
$G$ contains an odd minor $H$ of bounded size which does not contain
an odd-$K_k$-minor and has no $2496k$-coloring.
\end{enumerate}
\end{theorem}

Unfortunately this result does not give any information for minimal counterexamples to the odd Hadwiger's conjecture
for the case $t$.
Motivated by this fact, our purpose of this paper is to prove the following:

\begin{theorem}
\label{main1}
\showlabel{main1}
For a given graph $G$ and any fixed $t$, there is a polynomial time algorithm to
output one of the following:

\begin{enumerate}
\item
a $(t-1)$-coloring of $G$, or
\item
an odd $K_{t}$-minor of $G$, or
\item
after making all reductions to $G$, the resulting graph $H$ (which is an odd minor of $G$ and which has no reductions)
has a tree-decomposition $(T, Y)$ such that torso of each bag $Y_t$ is either
\begin{itemize}
\item
of size at most $f_1(t) \log n$ for some function $f_1$ of $t$, or
\item
a graph that has a vertex $X$ of order at most $f_2(t)$ for some function $f_2$ of $t$ such that $Y_t-X$ is bipartite.
Moreover,
degree of $t$ in $T$ is at most $f_3(t)$ for some function $f_3$ of $t$.
\end{itemize}
\end{enumerate}
\end{theorem}
Let us observe that the last odd minor $H$ is indeed a minimal counterexample to the odd Hadwiger's conjecture for the case $t$.
So we will show that any minimal counterexample to the odd Hadwiger's conjecture for the case $t$ satisfies the structure as in the third conclusion.
Here,
\DEF{torso} of a bag means that we can add all missing edges to $Y_t \cap Y_{t'}$ for all $tt' \in T$, but
then the resulting bags of the tree-decomposition still satisfy the third conclusion of Theorem \ref{main1}.

Let us point out that in order to obtain our theorem, we have to deal with the following structure:
\begin{quote}
$G$ has a vertex set of order at most $t-5$ such that $G-X$ is planar.
\end{quote}
Indeed, this structure appears in our proof. It follows from the Four Color Theorem that $G$ is $(t-1)$-colorable.

If we only obtain the first or the second conclusion of Theorem \ref{main1}, we can ``decide'' Conjecture \ref{odd} for the case $t$. However, if we obtain the third conclusion, then it is not clear
how to decide Conjecture \ref{odd} for the case $t$.  Indeed, suppose $G$ has a vertex $X$ of order at most $f_2(t)$ for some function $f_2$ of $t$ such that $G-X$ is bipartite. Then it is not so difficult to figure out whether or not
$G$ has an odd $K_t$-minor. (Essentially, we can reduce to the minor-testing algorithm \cite{RS13}. Let us give a
sketch of the proof.
If there is an odd $K_t$-minor in this graph $G$, then there are only constantly many possibilities for $X$ to be participated
in this odd $K_t$-minor. Hence we just need to figure out the existence of a ''label'' minor in $(G-X) \cup X'$
with respect to specified vertices in $X' \subseteq X$.)

But it is not clear if $G$ is $(t-1)$-colorable. Suppose $G-X$ has a bipartition $(A,B)$.
Then we can figure out the chromatic number of $X \cup A$ (and $X \cup B$).
This implies that we can color $G$ with at most $\chi(G)+1$ colors, where $\chi(G)$ is the chromatic number of $G$.
However, if the chromatic number of $X \cup A$ or $X \cup B$ is $t-1$, it is not clear if $G$ is $(t-1)$-colorable.

This is in contrast with Hadwiger's conjecture, which is decidable for the case $t$, as proved
by Robertson and Seymour (private communication, see also \cite{kt}), as follows:
\begin{theorem}
\label{main3}
\showlabel{main3}
For a given graph $G$ and any fixed $t$, there is a polynomial time algorithm to
output one of the following:

\begin{enumerate}
\item
a $(t-1)$-coloring of $G$, or
\item
a $K_{t}$-minor of $G$, or
\item
a minor $H$ of $G$ of tree-width at most $f(t)$ such that
$H$ does not have a $K_{t}$-minor nor is $(t-1)$-colorable.
\end{enumerate}
\end{theorem}

The last conclusion implies that $H$ is a counterexample to Hadwiger's conjecture of tree-width
at most $f(t)$ for the case $t$. So we can ``decide'' Hadwiger's conjecture for the case $t$.

Concerning Conjecture \ref{odd}, we can only give the following conclusion from Theorem \ref{main1}.

\begin{theorem}
\label{main2}
\showlabel{main2}
For a given graph $G$ and any fixed $t$, there is a polynomial time algorithm to
output one of the following:

\begin{enumerate}
\item
a $(t-1)$-coloring of $G$, or
\item
an odd $K_{t}$-minor of $G$, or
\item
after making all reductions to $G$, we can color the resulting graph $H$ (which is an odd minor of $G$ and which has no reductions) with at most $\chi(H)+1$ colors
in polynomial time.
\end{enumerate}
In the last conclusion, we can actually figure out whether or not
$H$ contains an odd $K_t$-minor.
\end{theorem}
\begin{proof}
We only need to deal with 3 of Theorem \ref{main1}. For each bag $Y_t$ that is ``nearly'' bipartite (i.e., the second outcome),
we take a vertex set $X'$ in $Y_t-X$ that is attached to a children bag. So $|X'| \leq f_3(t)$. We add the vertices in $X'$ to $X$.
Hence after this modifications, $Y_t-X$ is a bipartite graph $(A_{Y_t}, B_{Y_t})$
such that children bags of $Y_t$ are attached only to $X$. We delete $B_{Y_t}$ from $Y_t$. Let $Y''_t$ be the resulting bag.
Then $Y''_t$ consists of a tree-decomposition of width $f_2(t)+f_3(t)+1$
whose abstract tree is a star.
Let $Y'_t$ be the center bag of this tree-decomposition.
Let $\Y_2=\bigcup Y''_t$.

Let $\Y_1 =\bigcup Y_t$ where $Y_t$ is not nearly bipartite (hence $Y_t$ is not a bag that has $Y''_t$ in $\Y_2$). Let $\mathcal{Y'}=\Y_1 \cup \Y_2$.
Since each bag $Y''_t$ in $\Y_2$ is of tree-width
at most $f_2(t)+f_3(t)+1$ with the center bag $Y'_t$ and
children bags of $Y_t$ are attached only to $X$,
$\Y'$ induces a graph of tree-width at most  $f_1(t) \log n$ for some function $f_1$ of $t$ (with $f_1(t) \geq f_2(t)+f_3(t)+1$).

Thus we can construct a tree-decomposition of the graph $H'$ induced by $\mathcal{Y'}$  in polynomial time (see Theorem \ref{tree1} below) and then color the graph using the standard dynamic programming approach in polynomial time (see Theorem \ref{twcolor} below).
This allows us to color the graph $H'$ with $\chi(H')$ colors.

Now for each $Y_t$ having $(A_{Y_t},B_{Y_t})$ (i.e., $Y_t$ is ``nearly'' bipartite), we consider $B_{Y_t}$.
By our construction,
for any $t, t' \in T$ with $t \not=t'$, $B_{Y_t}$ has no neighbors
in $B_{Y_{t'}}$. Hence we need only one color for $\bigcup B_{Y_t}$.
So by adding $\bigcup B_{Y_t}$ (with one color) to $H'$ (with the coloring
using $\chi(H')$ colors), we obtain a $(\chi(H)+1)$-coloring of $H$,
as claimed in the third conclusion.

To figure out whether or not $H$ has an odd $K_t$-minor,
we can actually use the dynamic programming approach.
Indeed we can figure out whether or not $H'$ contains
an odd $K_t$-minor in polynomial time by the standard dynamic programming approach
(see Theorem \ref{oddparity} below). We then use this dynamic programming
information
to update the information for
each $Y_t$ having $(A_{Y_t},B_{Y_t})$. To do so, we just point out that
it is not hard to figure out whether or not $Y_t$ has an odd $K_t$-minor
by using the usual minor testing \cite{kkr,RS13}, see \cite{parity} (since $Y_t$ is nearly bipartite).
We have already gave a sketch of the proof for testing whether or not a nearly bipartite graph has an odd $K_t$-minor.

This allows us to update the dynamic programming information
for $Y_t$. The rest is the same as the standard dynamic programming
for the graphs of small tree-width (i.e., Theorem \ref{oddparity} below).
We omit the proof.
\end{proof}

Let us point out that our proof would be much simpler if we only consider Hadwiger's conjecture.
In this case, we only obtain the first structure of 3 in Theorem \ref{main1}, and
the large part of the proof in this paper would be gone. So our proof implies Theorem \ref{main3} with the bound on
tree-width $f(t)$ replaced by $f(t)\log n$, which is still enough to decide Hadwiger's conjecture for the case $t$ (because the standard dynamic programming approach (see Theorem \ref{twcolor} below) can handle graph coloring and minor testing for
a graph of tree-width $O(\log n)$).
This may be of independent interest.

It remains to show Theorem \ref{main1}.

\section{Overview of our main result}

We sketch a proof of Theorem \ref{main1}.

We may assume that
a given graph $G$ is a minimal counterexample to the odd Hadwiger's conjecture
for the case $t \geq 6$. For otherwise, we can perform a reduction. Then we apply the whole argument to the resulting graph.
So we now assume that there is no more reduction for $G$.

We first apply the above mentioned structure theorem for odd $K_t$-minor-free graphs \cite{DKM}.
Namely, every graph with no odd $K_t$-minor has a tree-decomposition $(T,Y)$ such that
each piece $Y_t$ is either
\begin{enumerate}
\item
after deleting bounded number $\alpha=\alpha(t)$ of vertices $Z_t$,
$Y_t-Z_t$ can be embedded in a surface $\Sigma$ of Euler genus $g$, up to 3-separations, with
at most $\alpha$ vortices $\V$ (c.f., $\alpha$-nearly embedded, see Section \ref{structthm} for more details), or
\item
after deleting bounded number $\alpha$ of vertices $Z_t$, $Y_t-Z_t$ is a bipartite graph.
\end{enumerate}

We can actually construct such a tree-decomposition $(T,Y)$ in polynomial time \cite{DKM}. If tree-width of $G$ is at most $f(\alpha) \log n$ (for some function of $\alpha$),
we are done. So assume that this is not the case.
%

There are two cases we have to consider:\par

\medskip

\DEF{Case 1.} There is a bag $Y_t$ which is nearly bipartite (i.e., the second).\par

\medskip

We can confirm that the number of children bags of $Y_t$ is bounded in terms of $\alpha$ (otherwise, we can perform a ``reduction'' of $G$, which implies that $G$ is no longer a minimal counterexample to the odd Hadwiger's conjecture for the case $t$, a contradiction). This indeed confirms the degree condition
on $Y_t$ in the second conclusion of 3 in Theorem \ref{main1}.\par

\medskip

\DEF{Case 2.} There is a  a bag $Y_t$ that is $\alpha$-nearly embedded (i.e., the first) .\par

\medskip

Let $G_0$ be the surface part of $Y_t-Z_t$.
Since $Y_t-Z_t$ is $\alpha$-nearly embedded, there may be some components $\W$ attached to faces of $G_0$, up to 3-separations.

We now divide $\W$ into two sets $\W_1$ and $\W_2$ such that
$\W_1$ consists of induced bipartite graphs, and $\W_2$ consists of
non bipartite graphs. Note that for each bipartite graph $W' \in \W_1$, $W'$ has to have at least $t-5$ neighbors in $Z_t$, for otherwise,
any $(t-1)$-coloring of $G-(W'-G_0)$ can be easily extended to a $(t-1)$-coloring of the whole graph $G$ (so we have a reduction, a contradiction).

It turns out that there are bounded number (say $g_1(\alpha)$) of disks $\D_1$ of bounded radius that cover all the components in $\W_2$ (otherwise, we can make a reduction of $G$, a contradiction).

We take a vertex set $U$ in $G_0$ such that each vertex in $U$ has at least $t-6$ neighbors in $Z_t$ (at least $t-5$ neighbors in $Z_t$ when $\Sigma$ is sphere).
We also take a set of odd faces $\O$ in $G_0$.

Let us assume the following(*):

Suppose there are at least $l \geq g_2(\alpha)$ faces in $G_0$ that
are pairwise ``far apart'',
such that each face either
\begin{enumerate}
\item
contains a vertex in $U$ or,
\item
accommodates a component in $\W_1$.
\end{enumerate}

\medskip

In this case, we can show that
there are bounded number (say $g_2(\alpha)$) of disks $\D_2$ of bounded radius that cover all odd faces in $\O$ in $G_0$. So in this case,
after deleting disks in $\D_1 \cup D_2$ and disks $\D_V$ that accommodate
vortices, the resulting graph $S$ in $G_0$ has no odd face.

With a little more work, we can, indeed, show that the resulting graph $S$ is bipartite (otherwise there is an odd $K_t$-minor).  Using this fact, we try to give an essentially 2-coloring of $S$, that
extends a $(t-1)$-coloring of the whole graph minus $S$. This is possible if we are allowed to precolor vertices that are very close to disks we deleted. More precisely,
we shall show that
\begin{quote}[(1)]
if there are vertices $R$ in the surface part of $G_0$ that are face-distance at least $f'(\alpha) \log n$ from all the disks we deleted,
for some function $f'$ of $\alpha$, then we can delete these vertices $R$ safely so that we can make a reduction (a contradiction).
More precisely, no matter how we color $G-R$ with $(t-1)$-colors, we can extend this coloring to a $(t-1)$-coloring of $R$ (hence a $(t-1)$-coloring of
the whole graph $G$, a contradiction).
\end{quote}

From this fact, with some work (using the result of Epstein \cite{epp1} which says that, for a planar graph $G$ with the outer face boundary $C$,
if there is no face that is of face-distance at least $\log n$, then
tree-width of $G$ is at most $6\log n$),
we can show the following:
\begin{quote}[(2)]
If tree-width of $Y_t$ is at least $f(t) \log n$, we can
perform a ``reduction'' for $G$ (i.e., $G$ is no longer a minimal counterexample to the odd Hadwiger's conjecture for the case $t$, a contradiction).
\end{quote}

Note that $\alpha$ depends on $t$. 
So it remains to consider the case when the assumption (*) is not satisfied. In this case, there are at most $l$ disks $\D_3$
that cover all the vertices in $U$ and all the components in $\W_1$.
So after deleting disks in $\D_1 \cup \D_3$ and disks $\D_V$ that accommodate
vortices, the resulting graph $S$ in $G_0$ is 2-cell embedded in a surface.
Using this fact, we try to give a 5-coloring $S$, that
extends a $(t-1)$-coloring of the whole graph minus $S$.
This is possible if we are allowed to precolor vertices that are very close to disks we deleted.
We also obtain the above (1) and (2) for this case too. The point here is that in $R$ every vertex has at most $t-7$ neighbors in $Z_t$ ($t-6$ neighbors in $Z_t$ if $\Sigma$ is sphere). This allows us to use the recent list coloring results for graphs in surfaces \cite{5list1,5list2,Th1}.

At the moment, either we can make a reduction (a contradiction) or we obtain the structure as in Theorem \ref{main1}.

\medskip


In the next section, we shall give some definitions.

\section{Preliminaries for the rest of the paper}

In this paper, $n$ and $m$ always mean the number of vertices of a given graph and
the number of edges of a given graph, respectively.
With a sloppy notation, for a graph $G_1$ and for a vertex set $Q$, we use $G_1 \cap Q$ to be the vertex of
the intersection.

We now look at definitions of the tree-width and the clique model.

\paragraph{Tree-width}

Let $G$ be a graph, $T$ a tree and let $Y = \{ Y_t
\subseteq V(G) \mid t \in V(T)\}$ be a family of vertex sets
$Y_{t}\subseteq V(G)$ indexed by the vertices $t$ of $T$. The pair
$(T,Y)$ is called a {\em tree-decomposition} of $G$ if it
satisfies the following three conditions:
\begin{itemize}
\item $V(G)=\bigcup_{t\in T}Y_{t}$,
\item for every edge $e\in E(G)$ there exists a $t\in T$ such that both ends of $e$ lie in $Y_{t}$,
\item
if $t, t', t'' \in V(T)$ and $t'$ lies on the path of $T$ between $t$ and $t''$,
then $Y_t \cap Y_{t''} \subseteq Y_{t'}$.
\end{itemize}
The {\em width} of $(T,Y)$ is the number $\max\{|Y_{t}|-1 \mid t\in T\}$ and the
{\em tree-width} ${\rm tw}(G)$ of $G$ is the minimum width of any tree-decomposition of $G$.
Sometime, we refer $Y_t$ to as a {\em bag}.

Robertson and Seymour developed  the first polynomial time algorithm for
constructing a tree decomposition of a graph of bounded tree-width \cite{RS13},
and eventually
came up with an $O(n^2)$ time algorithm for this problem.
Reed
\cite{rtree} developed an algorithm for the problem which runs in $O(n \log n)$ time, and then
Bodlaender \cite{bod} developed a linear time algorithm.
This algorithm was further improved in \cite{tree}.

\begin{theorem}
\label{tree1}
For an integer $w$,
there exists a $(2^{O(w)}) n^{O(1)}$ time algorithm that,
given a graph $G$,
either finds a tree-decomposition of $G$ of width $w$ or
concludes that the tree-width of $G$ is more than $w$.
Furthermore, if $w$ is fixed, there exists an $O(n)$ time algorithm to construct one of them.
\end{theorem}

We can apply dynamic programming to solve the graph coloring problem on graphs of
bounded tree-width, in the same way that we apply it to trees, provided that we are given  a bounded width tree
decomposition (see
e.g.~\cite{Arn}). Thus Theorem \ref{tree1} together with \cite{Arn} and \cite{Jansen} implies the following.

\begin{theorem}
\label{twcolor}
For integers $w$ and $k$,
there exists a $(2^{O(kw)}) n^{O(1)}$ time algorithm
to determine $\chi (G)$ in a graph $G$ of tree-width $w$.
Moreover, if $w$ and $k$ are fixed, there exists an $O(n)$ time algorithm to color $G$ using $\chi (G)$ colors.
\end{theorem}

By the same way, we can actually obtain the following too (see \cite{parity})

\begin{theorem}
\label{oddparity}
For integers $w$ and $k$,
there exists a $(2^{O(kw)}) n^{O(1)}$ time algorithm
to determine whether or not $G$ has an odd $K$-model in a graph $G$ of tree-width $w$ (where $|K|=k$).
Moreover, if $w$ and $k$ are fixed, there exists an $O(n)$ time algorithm for this problem.
\end{theorem}

\paragraph{Clique model}

For an integer $p$, $K_p$ is the complete graph with $p$ vertices.

A graph $G$ contains a \emph{$K_p$-model}
if there exists a function $\sigma$ with
domain $V(K_p) \cup E(K_p)$ such that
\begin{enumerate}
\item
for each vertex $v \in V(K_p)$, $\sigma(v)$ is a connected subgraph of $G$,
and the subgraphs $\sigma(v)$ ($v \in V(K_p)$) are
pairwise vertex-disjoint, and
\item
for each edge $e=uv \in E(K_p)$, $\sigma(e)$ is an edge $f \in
E(G)$, such that $f$ is incident in $G$ with a vertex in $\sigma(u)$
and with a vertex in $\sigma(v)$.
\end{enumerate}
Thus $G$ contains a $K_p$-minor if and only if $G$ contains a
$K_p$-model.
We call the subgraph $\sigma(v)$ ($v \in V(K_p)$) the
\emph{node} of the $K_p$-model.
The image of $\sigma$, which is a subgraph of $G$, is called the {\em $K_p$-model}.

We say that a $K_p$-model is \emph{even} if the union of the nodes
of the $K_p$-model consists of a bipartite graph. We also say that a
$K_p$-model is \emph{odd} if for each cycle $C$ in the union of the
nodes of the $K_p$-model, the number of edges in $C$ that belong to
nodes of the $K_p$-model is even. Thus if there is an odd $K_p$-model,
then there is a 2-coloring such that each node is bichromatic while each
edge joining two nodes is monochromatic.

\section{Structure Theorems}
\label{structthm}
\showlabel{structthm}

In the next section, we shall give a structure theorem for graphs without odd $K_t$-model.
In order to present that, we need to give several definitions concerning Robertson-Seymour's graph minor structure theorem.

\subsection{Tangles}

Let $G$ be a graph and let $A,B$ be subgraphs of $G$.
We say that the pair $(A,B)$ is a \DEF{separation} of $G$
if $A \cup B = G$, $V(A)-V(B)\not=\emptyset$, and
$V(B)-V(A)\not=\emptyset$. The \DEF{order} of a separation
$(A,B)$ is $|V(A) \cap V(B)|$.

A \emph{tangle} of order $k$ of $G$ is a set $\mathfrak T$ of
separations of $G$ of order $<k$ satisfying the following there
conditions.
\begin{enumerate}
\item For all separations $(A,B)$ of $G$ of order $<k$, either
  $(A,B)\in\mathfrak T$ or  $(B,A)\in\mathfrak T$.
\item If $(A_1,B_1),(A_2,B_2),(A_3,B_3)\in\mathfrak T$ then $A_1\cup
  A_2\cup A_3\neq G$.
\end{enumerate}

Note that if $(A,B)\in\mathfrak T$  then $(B,A)\notin\mathfrak T$; we think of $B$ as the `big side' of the separation~$(A,B)$, with respect to this tangle (and similarly $A$ is the ``small side'').

Let $\mathfrak T$ be a tangle of order at least $p$. We say that a $K_p$-model is \emph{controlled by the tangle $\mathfrak T$} if
no node of the $K_p$-model is contained in $A-B$ of any separation $(A,B)\in\mathfrak T$ of order
at most $p-1$.
%
%
%

\subsection{Societies and Vortices}

A \emph{society} is a pair $(G,\Omega)$, where $G$ is a graph and
$\Omega$ a cyclic permutation of a subset $V(\Omega)$ of $V(G)$ (we call $V(\Omega)$ \emph{society vertices}). Note that
for every $w\in V(\Omega)$ we have $V(\Omega)=\{\Omega^j(w)\mid 0\le
j<|V(\Omega)|\}$. The
\emph{length} of a society $(G,\Omega)$ is $|V(\Omega)|$.

A society $(G,\Omega)$ of length $\ell$ is a \emph{$\rho$-vortex} if for all
$w\in{V(\Omega)}$ and $k\in[\ell]$ there do not exist
$(\rho+1)$ mutually disjoint paths of $G$ between $\{\Omega^j(w)\mid
0\le j<k\}$ and $\{\Omega^j(w)\mid k\le j<\ell\}$.

A \emph{linear decomposition} of a society $(G,\Omega)$ of length $\ell$ is a sequence
$(X_i)_{0\le i<\ell}$ of subsets of $V(G)$ such that
\begin{enumerate}
\item $\bigcup_{i=0}^{\ell-1} X_i=V(G)$.
\item $X_i\cap X_k\subseteq X_j$ for
$0\le i\le j\le k<\ell$.
\item There is a $x_0\in{V(\Omega)}$ such that
$\Omega^i(x_0)\in X_i$ for $0\le i<\ell$.
\end{enumerate}
The \emph{width} of the linear decomposition
$(X_i)_{0\le i<\ell}$ is $\max\{|X_i|\mid0\le i<\ell\}$, and the
\emph{depth} of $(X_i)_{0\le i<\ell}$ is $\max\{|X_i\cap
X_{i+1}|\mid0\le i<\ell-1\}$. Sometimes $X_i$ is called a \emph{bag} (of a linear decomposition of a society $(G,\Omega)$).

The following  is proved in \cite{RS9}:

\begin{lemma}\label{lem:omega}\showlabel{lem:omega}
 If a society $(G,\Omega)$ is an $\alpha$-vortex
then it has a linear decomposition of depth at most $\alpha$. Moreover,
such a decomposition can be found in $O(n^2)$ time.
\end{lemma}

Note that the algorithmic result follows from the proof in \cite{RS9}
(or see \cite{RS} for an easier description of the algorithm).
A stronger result is given in \cite{newgm} too.

\subsection{Near Embeddings}

Robertson and Seymour's main theorem is concerning the structure
capturing a big side with respect to a tangle. We now mention one version of their result. We follow the notations in \cite{DKMW}, but let us
repeat some of them for a self-contained reason.

For a positive integer $\alpha$, a graph $G$ is
\textit{$\alpha$-nearly embeddable} in a surface $\Sigma$ if there
is a subset $Z\subseteq V(G)$ with $|Z|\leq \alpha$ (i.e., the apex vertex set), two sets
$\mathcal V=\{(G_1,\Omega_1),\ldots,(G_{\alpha'},\Omega_{\alpha'})\}$, where $\alpha'\le\alpha$, and $\mathcal W=\{(G_{\alpha'+1},\Omega_{\alpha'+1}),\ldots,(G_{n},\Omega_{n})\}$ of societies, and a graph $G_0$ such that
the following conditions are satisfied.
\begin{enumerate}
\item $G-Z=G_0\cup G_1\cup\ldots\cup G_n$.
\item  For all  $1\leq i\leq j \leq n$ and $\Omega_i:= V(G_i\cap G_0)$, the pairs $(G_i,\Omega_i)=:V_i$ are vortices and $G_i\cap G_j \subseteq G_0$ when $i\neq j$ .
\item The vortices $V_1,\ldots,V_{\alpha'}$ are disjoint and have adhesion at most $\alpha$; we denote the set of these vortices by $\V$. 
We will sometimes refer to these vortices as \emph{large} vortices.
\item The vortices $V_{\alpha'+1},\ldots,V_n$ have length at most 3; we denote the set of these vortices by $\W$. 
These are the \emph{small} vortices of the near-embedding.
\item
There are closed discs in~$\Sigma$ with disjoint interiors $D_1,\ldots, D_n$ and an embedding ${\sigma: G_0 \hookrightarrow \Sigma-\bigcup_{i=1}^n D_i}$ such that $\sigma(G_0)\cap\boundary D_i = \sigma(\Omega_i)$ for all~$i$ and the generic linear ordering of $\Omega_i$ is compatible with the natural cyclic ordering of its image (i.e., coincides with the linear ordering of $\sigma(\Omega_i)$ induced by $[0,1)$ when $\boundary D_i$ is viewed as a suitable homeomorphic copy of $[0,1]/\{0,1\}$). For $i=1,\dots,n$ we think of the disc $D_i$ as \emph{accommodating} the (unembedded) vortex~$V_i$, and denote $D_i$ as~$D(V_i)$.
\end{enumerate}
We call $(\sigma,G_0,Z,\mathcal V,\mathcal W)$ an \emph{$\alpha$-near
  embedding} of $G$ in~$\Sigma$ or just \emph{near-embedding} if the
bound is clear from the context.

Let $G'_0$ be the graph resulting from $G_0$ by joining any two nonadjacent vertices $u,v\in G_0$ that lie in a common vortex $V\in\W$; the new edge $uv$ of $G'_0$ will be called a \emph{virtual edge}. By embedding these virtual edges disjointly in the disks $\Delta$ accommodating their vortex~$V$, we extend our embedding $\sigma\colon G_0\hookrightarrow\Sigma$ to an embedding $\sigma'\colon G'_0\hookrightarrow\Sigma$. We shall not normally distinguish $G'_0$ from its image in $\Sigma$ under~$\sigma'$.

\drop{
A vortex $(G_i,\Omega_i)$ is \emph{properly attached} to $G_0$ if it satisfies the following two requirements. First, for every pair of distinct vertices $u,v\in \Omega_i$ the graph $G_i$ must contain a path (one with no inner vertices in~$\Omega_i$) from $u$ to~$v$. Second, whenever $u,v,w\in\Omega_i$ are distinct vertices (not necessarily in this order), there are two internally disjoint paths in $G_i$ linking $u$ to~$v$ and $v$ to~$w$, respectively.

A near-embedding $(\sigma,G_0,Z,\mathcal V,\mathcal W)$ is \emph{nice}
if for all $(G_i,\Omega_i)\in\mathcal V$ there is a cycle
$C_i\subseteq G_0$ such that $\sigma(C_i)$ is the
boundary of the disk $\Delta_i$.}

A near-embedding $(\sigma,G_0,Z,\mathcal V,\mathcal W)$ is
controlled by a tangle \emph{$\mathfrak T$}, if
for all $(H,\Omega)\in\mathcal V\cup\mathcal W$ there is no
$(A,B)\in\mathfrak T$ with $Z\subseteq V(A\cap B)$ and $B\setminus Z\subseteq H$.

%

A cycle $C$ in $\Sigma$ is \emph{flat} if $C$ bounds an open disk
$D(C)$ in $\Sigma$. A flat triangle is a \emph{boundary triangle} if
it bounds a disk that is a face of $G'_0$ in~$\Sigma$. Disjoint
cycles $C_1,\ldots,C_n$ in $\Sigma$ are \emph{concentric} if they
bound disks $D_1 \supseteq \ldots \supseteq D_n$ in~$\Sigma$. A path
system~$\P$ (i.e, a set of disjoint paths) intersects
$C_1,\ldots,C_n$ \emph{orthogonally} if every path $P$ in $\P$
intersects each of the cycles in a (possibly trivial) subpath of
$P$.

For a near-embedding $(\sigma,G_0,A,\V,\W)$ of some graph $G$ in a surface $\Sigma$ and a vortex  $V\in\V$,  let $C_1,\ldots,C_n$ be cycles in $G'_0$ that are concentric in $\Sigma$. The cycles $C_1,\ldots,C_n$ \emph{enclose} $V$ if $D(C_n)\setminus \boundary D(C_n)$ contains~$\Omega(V)$. They  \emph{tightly enclose} $V$ if the following holds:
\begin{equation*}
\begin{minipage}[c]{0.8\textwidth}\em
For every vertex $v\in V(C_k)$ and for all $1\leq k \leq n$, there is a vertex $w\in\Omega(V)$
such that the face distance of $v$ and $w$ in $\Sigma$ is at most $n-k+2$.
\end{minipage}\ignorespacesafterend
\end{equation*}

For positive integers~$r$, define a graph $H_r$ as follows.  Let $P_1, \dots, P_r$
be $r$ vertex disjoint (`horizontal') paths of length $r-1$, say $P_i = v_1^i\dots
v_r^i$. Let $V(H_r) = \bigcup_{i=1}^r V(P_i)$, and let
\begin{equation*}
\begin{split}
E(H_r) = \bigcup_{i=1}^r E(P_i) \cup
\Big\{&v_j^i v_j^{i+1} \mid \text{ $i,j$ odd};\ 1 \le i < r;\ 1 \le j \le r\Big\} \\
& \cup \Big\{v_j^i v_j^{i+1} \mid \text{ $i,j$ even};\ 1 \le i < r;\ 1 \le j \le r\Big\}.
\end{split}
\end{equation*}
\drop{
We call the paths $P_i$ the \emph{rows} of $H_r$; the paths induced by the vertices $\{v^i_j,v^i_{j+1}:1\leq i\leq r\}$ for an odd index $i$ are its \emph{columns}.
}
The 6-cycles in $H_r$ are its \emph{bricks}. In the natural plane embedding of~$H_r$, these bound its `finite' faces. The outer cycle of the unique maximal 2-connected subgraph of $H_r$ is the \emph{boundary cycle}  of $H_r$.

Any subdivision $H = T H_r$ of $H_r$ will be called an \emph{$r$-wall} or a \emph{wall of size $r$}. The \emph{bricks} and the \emph{boundary cycle} of $H$ are its subgraphs that form subdivisions of the bricks and the boundary cycle of~$H_r$, respectively. An embedding of $H$ in a surface~$\Sigma$ is a \emph{flat} embedding, and $H$ is \emph{flat} in~$\Sigma$, if the boundary cycle $C$ of $H$ bounds a disk $D(H)$ that contains a vertex of degree 3 of $H-C$.

A closed curve $C$ in $\Sigma$ is \emph{genus-reducing} if the (one or two) surfaces obtained by capping the holes of the components of $\Sigma\setminus C$ have smaller genus than $\Sigma$. Note that if $C$ separates $\Sigma$ and one of the two resulting surfaces is homeomorphic to $S^2$, the other is homeomorphic to~$\Sigma$. Hence in this case $C$ is not genus-reducing.

The \emph{representativity} of an embedding $G\hookrightarrow\Sigma\not\simeq S^2$ is the smallest integer $k$ such that every genus-reducing curve $C$ in $\Sigma$ that meets $G$ only in vertices meets it in at least $k$ vertices.

An $\alpha$-near embedding $(\sigma,G_0,A,\V,\W)$ of a graph $G$ in
some surface $\Sigma$ is {\it $\delta$-rich} for some integer
$\delta$ if the following statements hold:
\begin{enumerate}[(i)]
\item  \label{prop:bigwall}
$G'_0$ contains a flat $r$-wall $H$ for an integer $r\geq  \delta$.
\item
The representativity of $G'_0$ in $\Sigma$ is at least $\delta$.
\item
For every vortex $V\in \V$ there are $\delta$ concentric cycles
$C_1(V),\ldots,C_{\delta}(V)$ in $G'_0$ tightly enclosing $V$ and
bounding open disks $D_1(V) \supseteq \ldots \supseteq D_\delta(V)$,
such that $D_\delta(V)$ contains $\Omega(V)$ and $D(H)$ does not
intersect $D_1(V)\cup C_1(V)$. For distinct vortices ${V,W\in\V}$,
the disks $\closure{D_1(V)}$ and $\closure{D_1(W)}$ are disjoint.
(Sometime, the cycle $C_1(V)$ is called the outermost cycle, and the cycle $C_{\delta}(V)$ is
called the innermost cycle.)
\item
Every two vortices in $\V$ have distance at least $\delta$ in
$\Sigma$.
\item \label{prop:linkagetogrid}
For every vortex $V\in\V$, its set of society vertices $\Omega(V)$
is linked in $G'_0$ to branch vertices of $H$ by a path system
$\P(V)$ of $\delta$ disjoint paths having no inner vertices in $H$.
\item \label{prop:orthogonalpaths}
For every vortex $V\in\V$, the path system $\P(V)$ intersects the
cycles $C_1(V),\ldots,C_\delta(V)$ orthogonally.
\item
All vortices in $\W$ are properly attached to $G_0$.
\end{enumerate}\bigbreak

We are now ready to state the Robertson and Seymour's main theorem,
Theorem (3.1), in \cite{RS16}. This theorem is concerning the
structure relative to big sides of separations of small order, with respect to a given tangle $\mathfrak T$ of large
order. Actually, we use the following more subtle version of Theorem
(3.1) in \cite{RS16}, which is shown in \cite{DKMW}.

\begin{theorem}\label{thm:3.1}
For every graph~$R$ there is an integer~$\alpha$ such that for every
integer~$\delta$\footnote{The proof in \cite{DKMW} can be easily modified so that $\delta$ can depend on $\alpha$. This will be used in our proof. See Section \ref{smallv} and thereafter.} there is an integer $w=w(R,\delta)$ such that the
following holds. Every graph~$G$ with a tangle $\mathfrak T$ of order at
least $w$ that does not contain $R$ as a minor has an $\alpha$-near,
$\delta$-rich embedding in some surface~$\Sigma$ in which $R$ cannot
be embedded. Moreover, this $\alpha$-near, $\delta$-rich embedding
is controlled by this tangle $\mathfrak T$.
\end{theorem}

Given a tangle $\mathfrak T$, a polynomial time algorithm to construct one of the conclusions in Theorem \ref{thm:3.1} is given in \cite{decomposition,newgm,short}.

\section{Structure theorem for graphs without an odd $K_t$-minor}
\label{sec:struc}
\showlabel{sec:struc}

We now give a structure theorem for graphs without odd $K_k$-model.

\begin{theorem}\label{thm:extended}\showlabel{thm:extended}
For every integer $k$ and every integer $\delta$
there exist integers $\alpha$ and  $\theta$ such that for every
graph $G$ that does not contain an odd $K_k$-model and every $Z\subseteq
V(G)$ with $|Z|\leq 3\theta-2$ there is a rooted tree-decomposition (with a rooted tree $T$ and root $t''$)
$(V_t)_{t\in T}$ of $G$ such that for every $t\in T$, either
\begin{enumerate}
\item
there is a vertex set $Z'_t$ (an apex set) of order at most $\alpha$ such that $V_t-Z'_t$ induces a bipartite graph, and moreover,
$|V(t) \cap V(t') -Z'_t| \leq 1$ for each $tt' \in T$ where $t'$ is a children of $t$, or
\item
there is a
surface $\Sigma_t$ of Euler genus $\alpha$,  and the torso
of $G_t$ (i.e, obtained from the graph induced by $V_t$ by
making all $G_t \cap G_{t'}$ cliques for $tt' \in T$, where $t'$ is a children of $t$) has an
$\alpha$-near, $\delta$-rich, embedding $(\sigma_t,G_{t,0},Z'_t,\V_t,\emptyset)$ into
$\Sigma_t$ with the following properties:
\begin{enumerate}
\item All vortices have linear decompositions of width at most $\alpha$.
\item For every $t'\in T$ with $tt'\in E(T)$ there is a vertex set $X$ which is either
\begin{enumerate}
\item two consecutive parts of an $\alpha$-vortex or
\item a subset of $V(G_{t,0})$ that induces in $G_{t,0}$ a $K_1$, a $K_2$ or a boundary triangle (i.e., it bounds a disk that is a face of $G'_{t,0}$ in~$\Sigma$.)
\end{enumerate}
such that $V_t\cap V_{t'}\subseteq X\cup Z'_t$.
\end{enumerate}
\end{enumerate}
Further, $Z\subseteq Z'_{t''}$.
%
%
\end{theorem}

In addition, given $k$, we can find either an odd $K_k$-model or such a tree-decomposition  in polynomial time \cite{DKM}. For the completeness,
we shall give a proof of Theorem \ref{thm:extended} in the appendix.

\section{List-Coloring Extensions in Planar Graphs}
\label{coloring extensions}
\showlabel{coloring extensions}

%

Let $G$ be a graph. A \DEF{list-assignment} is a function $L$ which assigns
to every vertex $v\in V(G)$ a set $L(v)$ of
natural numbers, which are called \DEF{admissible colors} for that vertex.
An \DEF{$L$-coloring} of $G$
is an assignment of admissible colors to all vertices
of $G$, i.e., a function $c:V(G)\to \NN$ such that $c(v)\in L(v)$ for every
$v\in V(G)$, and for every edge $uv$ we have $c(u)\ne c(v)$. If $k$ is an integer
and $|L(v)|\ge k$ for every $v\in V(G)$, then $L$ is a \DEF{$k$-list-assignment}.
The graph is \DEF{$k$-list-colorable} (or \DEF{$k$-choosable})
if it admits an $L$-coloring for every
$k$-list-assignment $L$. If $L(v)=\{1,2,\dots,k\}$ for every $v$, then
every $L$-coloring is referred to as a \DEF{$k$-coloring} of $G$. If $G$
admits an $L$-coloring ($k$-coloring), then we say that $G$ is \DEF{$L$-colorable}
(\DEF{$k$-colorable}).

The smallest integer $k$ such that $G$ is $k$-choosable
is the \DEF{list-chromatic number} $\chi_l(G)$.

In \cite{Th1}, Thomassen proved a result which is slightly stronger than
the statement that every planar graph is 5-list-colorable.
This is used in our proof.

\begin{theorem}[Thomassen \cite{Th1}]
\label{thm:Th5}
\showlabel{thm:Th5}
Let $G$ be a plane graph with outer facial walk $C$, and let $x,y$ be
adjacent vertices on $C$. Let $L$ be a list-assignment for $G$ such that
$L(x)=\{\a\}$, $L(y)=\{\b\}$, where $\b\ne \a$, every vertex on
$C\setminus\{x,y\}$ has at least three admissible colors, and every
vertex that is not on $C$ has at least five admissible colors.
Then $G$ can be $L$-colored.
\end{theorem}

Theorem \ref{thm:Th5} has the following useful corollary.

\begin{corollary}
\label{cor:Th5}
\showlabel{cor:Th5}
Let $G$ be a plane graph whose outer face $C$ is either a 3-face or a 4-face.
Let $L$ be a $5$-list-assignment for $G$. Then every $L$-coloring
of $C$ can be extended to an $L$-coloring of the whole graph $G$.
\end{corollary}

A proof is given in \cite{DKM}.

\drop{
Moreover, the following lemma, which is very useful for
our proof, was proved in \cite{DKM}.

\begin{lemma}
\label{lem:non5}
\showlabel{lem:non5}
Let $G$ be a plane graph whose outer face $C$ is a cycle of length at most $6$.
Let $L$ be a $5$-list-assignment for $G$.
Then every $L$-coloring
of $C$ can be extended to an $L$-coloring of the whole graph $G$ unless
$G$ has a subgraph
as shown in
Figure \ref{fig:4} such that its outer face boundary is $C$.
\end{lemma}

\begin{figure}[htb]
\epsfxsize5cm
\centerline{\epsffile{5list4.eps}}
\caption{Short contractible cycles}
\label{fig:4}
\end{figure}

Let us observe that all the graphs shown in Figure \ref{fig:4}, except for the second one,
have a vertex of degree at most
5 inside the outer face boundary $C$.}

\section{Minimal counterexample to the odd Hadwiger's conjecture}

In this section, we shall look at minimal counterexamples to the odd Hadwiger's conjecture, and
give some basic properties about them. Let us remind some notations.

Let $G$ be a graph satisfying the following conditions:
\begin{enumerate}
\item[\rm (1)]
$G$ is $t$-chromatic.
\item[\rm (2)]
$G$ is minimal with respect to the odd-minor-operation in the class of all
$t$-chromatic graphs.
\item[\rm (3)]
$G$ does not contain $K_t$ as a minor.
\end{enumerate}


We call such a graph $G$ a {\it minimal counterexample to the odd Hadwiger's conjecture} for the case $t$.
The odd Hadwiger's conjecture suggests there are {\it no} such graphs.
It is easy to see that $G$ has minimum degree at least $t-1$
and we generalize this result in Lemma \ref{mind1}.
We only consider cases $t \geq 6$, as other cases are already solved (see Subsection 1.2).

We need the following easy fact, whose proof is omitted.
\begin{fact}
\label{easy1}
\showlabel{easy1}
Let $u,v$ be two vertices of a given graph $G$. If $G$ does not contain an odd path with two endvertices $u,v$ and $G-\{u,v\}$  is connected,
then $G-\{u,v\}$ can be written as $W \cup B_1 \cup \dots \cup B_l$ (for some $l$) such that
$W$ induces a bipartite graph $(A, B)$ in $G$, $B_i$ is not bipartite, $|B_i \cap W| \leq 1$ for each $i$, for any $i, j$
with $i \not=j$ $|B_i \cap B_j| \leq 1$, and both $u$ and $v$ have neighbors only in $A$ or $B$.
\end{fact}

Here are a few easy facts. So we omit the proofs.

\begin{fact}
\label{easy2}
\showlabel{easy2}
Let $(A, B)$ be a separation of a given graph $G$ and let $u, v \in A \cap B$.
If $uv \not\in E(G)$, then $B-A$ can be reduced onto $A \cap B$, via the odd-minor-operations,
with adding the edge $uv$,
unless there is no odd path between $u$ and $v$ in $(B-A) \cup \{u,v\}$, in which case,
$(B-A) \cup \{u,v\}$ can be written as $W \cup B_1 \cup \dots \cup B_l$ (for some $l$) such that
$W$ induces a bipartite graph $(A', B')$, $B_i$ is not bipartite, $|B_i \cap W| \leq 1$ for each $i$, for any $i, j$ with $i \not=j$
$|B_i \cap B_j| \leq 1$, and both $u$ and $v$ have neighbors only in $A'$ or $B'$.
\end{fact}

\begin{fact}
\label{easy3}
\showlabel{easy3}
Let $(A, B)$ be a separation of a given graph $G$ and let $u, v \in A \cap B$.
If $uv \not\in E(G)$, then $B-A$ can be reduced onto $A \cap B$, via the odd-minor-operations,
with identifying $u$ and $v$,
unless there is no even path between $u$ and $v$ in $(B-A) \cup \{u,v\}$, in which case, $(B-A) \cup \{u,v\}$ can be written as $W \cup B_1 \cup \dots \cup B_l$ (for some $l$) such that
$W$ induces a bipartite graph $(A',B')$, $B_i$ is not bipartite, $|B_i \cap W| \leq 1$ for each $i$, for any $i, j$ with $i \not=j$
$|B_i \cap B_j| \leq 1$, and $u$ has neighbors only in $A'$ ($B'$, resp.), but $v$ has neighbors only in $B'$ ($A'$, resp.)
\end{fact}

\begin{lemma}
\label{mind1}
\showlabel{mind1}
Let $G$ be a minimal counterexample to the odd Hadwiger's conjecture for the case $t$. Then minimum degree is at least $t$. In addition, if $d(v) =(t-2)+l$ for some vertex $v \in V(G)$, then there is no
vertex set $N' \subset N(v)$ with at least $l+1$ vertices that are independent.
\end{lemma}
\begin{proof}
Let $v$ be a vertex of degree $< t$ in $G$. Since $G-v$ has a $(t-1)$-coloring $\phi$, so if $d(v) < t-1$, then we can clearly put $v$ back to $G$ so that
the coloring $\phi$ can be extended to $v$ to give a $(t-1)$-coloring of $G$ (a contradiction).
So it remains to consider the case when $d(v) \geq t-1$. We may also assume that $N(v)$ does not induce a complete graph.

Consider the case when $d(v)=t+l-2$ and there are $l+1$ vertices $v_1, \dots, v_{l+1} \in N(v)$ such that $v_1, \dots, v_{l+1}$ are independent in $G$.
We delete all edges incident with $v$, except for $vv_1, \dots, vv_{l+1}$. Let $G'$ be the resulting graph. Now $vv_1, \dots, vv_{l+1}$ consist of
a cut in $G'$ such that one side only contains $v$. Contract $v, v_1, \dots, v_{l+1}$ into a single vertex $v'$.
These operations are odd-minor-operations, so
if the resulting graph $G''$ has an odd $K_t$-minor, so does $G$. So we may assume that $G''$ has a $(t-1)$-coloring $\phi'$.
We can now extend the coloring $\phi'$ to a $(t-1)$-coloring of $G$
because $v_1, \dots, v_{l+1}$ consist of an independent set (and hence
they receive the same color from $\phi'$), and $v$ only sees at most $(t-2)$ colors of $\phi'$.

This implies the second conclusion, and the first conclusion also follows because if $d(v)=t-1$, then  $N(v)$ does not induce a complete graph, but
then there are two vertices $v_1,v_2 \in N(v)$ with $v_1v_2 \not\in E(G)$ (so $v_1, v_2$ are independent).
\end{proof}

The following is straightforward too.
\begin{lemma}
\label{conn1}
\showlabel{conn1}
Let $G$ be a minimal counterexample to the odd Hadwiger's conjecture for the case $t$.  Then there is no separation $(A, B)$ of order at most $t-3$ such that
one of $A-B$, $B-A$ is bipartite.
\end{lemma}

We now give a fundamental lemma that tells one structural property of $G$.

\begin{lemma}
\label{minimal}
\showlabel{minimal}
Let $G$ be a minimal counterexample to the odd Hadwiger's conjecture for the case $t$.
Then $G$ does not have a separation $(A,B)$ with the following property:

 For any partition of $A \cap B$ into independent sets $V_1, \dots, V_l$ with $l \leq t-1$, $B$ can be reduced onto $A \cap B$, via the odd-minor-operations, such that,
all $V_i$ are reduced into a single vertex in the resulting graph, and the resulting graph on $A \cap B$ is a clique.
\end{lemma}
\begin{proof}
We first color $B$. Note that $A-B\not=\emptyset$, so $B$ is no longer a counterexample.
Hence $B$ must have a $(t-1)$-coloring $\phi$ (otherwise, $B$ has an odd $K_t$-minor).

This coloring partitions the vertices of $A \cap B$ into color classes $V_1, \dots, V_l$ with $l \leq t-1$.
By the assumption of the lemma, we can reduce $B$ onto $A \cap B$, via the odd-minor-operations, such that
all $V_i$ are reduced into a single vertex, and the resulting graph on $A \cap B$ is a clique. Let $G'$ be the resulting graph of $G$.

By the minimality of $G$, $G'$ has a $(t-1)$-coloring, and this coloring
together with $\phi$ gives rise to a $(t-1$)-coloring of $G$, a contradiction.
\end{proof}

The following lemma is easy to show (so we omit the proof).

\begin{lemma}
\label{re1}

Let $S$ be a vertex set of order $t$.
Suppose there are components $C_1, \dots, C_l$ in $G-S$ such that
each $C_i$ has the following property:

Each $C_i$ has neighbors to all the vertices in $S$.
Moreover, for any two vertices $u, v \in S$, there is an odd path between $u$ and $v$ in
$C_i \cup \{u,v\}$

If $l \geq t^2/2$, then there is an odd $K_t$-minor.
\end{lemma}

We give the following straightforward lemma as well.

\begin{lemma}
\label{re2}

Let $S$ be an independent set of order $s$.
Suppose there are connected components $C_1, \dots, C_l$ in $G-S$ such
that each $C_i$
has neighbors to all the vertices in $S$. Also, suppose furthermore that,
for any two vertices $u, v \in S$, there is an even path between $u$ and $v$ in
$C_i \cup \{u,v\}$.

If $l \geq s$, then $S$ can be contracted into a single vertex, via the odd-minor-operations, using the components $C_1, \dots, C_l$.
\end{lemma}

We also need the following lemma, whose proof follows from Lemma \ref{conn1}.

\begin{lemma}
\label{re3}
\showlabel{re3}
Let $(A,B)$ be a separation of order at most $t-2$ in a given graph $G$.
Suppose that $A-B$ is an induced bipartite graph.

No matter how we $(t-1)$-color $B$, we can extend this coloring
to a $(t-1)$-coloring of $G$, unless there is a vertex in $A-B$ that is adjacent to all vertices of $A \cap B$ and $|A \cap B|=t-2$.

%
%
%
\end{lemma}
\begin{proof}
If $|A \cap B| \leq t-3$, the result follows from Lemma \ref{conn1}.
Assume $|A \cap B| = t-2$. If $A \cap B$ uses at most $t-3$ colors of the $(t-1)$-coloring of $B$, we can also extend the coloring of $B$ to a $(t-1)$-coloring of $G$ because $A-B$ needs only two colors.
So it remains to consider the
case when $|A \cap B|=t-2$ and $A \cap B$ uses exactly $t-2$ colors of the coloring of $B$.

Since there is exactly one color $a$ that is not used in the coloring of $A \cap B$, we first color one partite set $B_1$ of the bipartite graph $(B_1,B_2)$ of $A-B$ with the color $a$. If there is no vertex in $B_2$ that is adjacent to all the vertices of
$A \cap B$, for each vertex in $B_2$ there is always one color that yields a valid coloring. Thus we can extend the coloring of $B$ to the whole graph $G$, unless there
is a vertex in $B_2$ that is adjacent to all the vertices in $A \cap B$.
\drop{
The second statement holds by just following the above proof. More precisely, if there is a color $a$ that is not used in the coloring of $A \cap B$,
then we first color one partite set $B_1$ of the bipartite graph $(B_1,B_2)$ of $A-B$ with the color $a$, where $B_1$ does not contain $v$.
Then we color $B_2$ with the same color as the one $v$ has received from the coloring of $B$.
If all $(t-1)$-colors are used in $A \cap B$, we appoint the color $a$ as the color of $v$ (from the coloring of $B$).
The rest of the argument is exactly same, so we omit the proof.}
\end{proof}

We need the following variants of Lemma \ref{re3}. This lemma indeed confirms the degree condition
on $Y_t$ in the second conclusion of 3 in Theorem \ref{main1}.

\begin{lemma}
\label{re4}
\showlabel{re4}
Let $G$ be a minimal counterexample to the odd Hadwiger's conjecture
for the case $t$.
Suppose $G-Z$ can be written as $W \cup B_1 \cup \dots \cup B_l$ (for some $l$) such that
$W$ induces a bipartite graph $(A',B')$, $B_i$ is not bipartite, $B_i-W$ is connected, $|B_i \cap W| \leq 1$ for each $i$, and for any $i, j$ with $i \not=j$,
$|B_i \cap B_j| \leq 1$. Then $l < |Z|^22^{|Z|}$.
\end{lemma}
\begin{proof}
Suppose $l \geq |Z|^22^{|Z|}$. Then by the pigeon hole principle, there is a set $\Q'$ of $B_1,\dots, B_l$ with
$|\Q'| \geq |Z|^2$ such that for any two sets $Q,Q' \in \Q'$, $N(Q) \cap Z =N(Q') \cap Z = Z' \subset Z$.

We now reduce $Z'$ into a clique $Z''$, via odd-minor-operations using
some components in $\Q'$.
We first keep identifying two vertices $u,v \in Z'$ by taking an even path
in $(Q-\{y\}) \cup \{u,v\}$ between $u$ and $v$, if $uv$ is not present, where $y =Q \cap (G-Z)$.

When we stuck, we still have at least $|Z|^2-|Z|$ components $Q$ remaining in $\Q'$. For each of these components,
since we cannot identify two vertices of $Z'$ via odd-minor-operations, we can add an edge $uv$ (with $u,v \in Z'$) from $Q \in Q'$ by taking an odd path in $(Q-\{y\}) \cup \{u,v\}$ between $u$ and $v$, if $uv$ is not present, where $y =Q \cap (G-Z)$.
This way, we can reduce $Z'$ onto a clique $Z''$ via odd-minor-operations.
Let $\Q''$ be the components in $\Q$ that are used to construct the clique $Z''$. Note that $\Q'-\Q''\not=\emptyset$.
If $|Z''| \geq t$, we are done, as this is an odd $K_t$-model.

Let $G'$ be the resulting graph. Note that
none of the components in $\Q''$ exists in $G'$.
By the minimality of $G$, $G'$ has a $(t-1)$-coloring $\phi$.
For each $Q \in \Q''$ with $Q \cap (G-Z)=\{w\}$, $w$ gets a color from $\phi$.
If the color of $w$ is used in the coloring of $Z''$, say $w$ and $x$ in $Z''$ receive the same color,
then we just take an even path from $w$ to $x$ via some component $Q'$ in $\Q'-\Q''$ (which is possible because $Q'$ is not bipartite).
This allows us to identify $w$ and $x$ via the odd-minor-operations
in $G-Q$. Then we use components in $\Q'-\{Q,Q'\}$ to reduce to
the clique of $Z''$, via the odd-minor-operations (which is possible, because all components in $\Q'-\{Q,Q'\}$ are not bipartite and
$|\Q'| \geq |Z|^2$).

If the color of $w$ is different from any of the colors in $Z''$, we use components in $\Q'-\{Q\}$ to reduce to
the clique of $\{w\} \cup Z''$, via the odd-minor-operations (which is possible, again, because all components in $\Q''-\{Q\}$ are not bipartite and $|\Q'| \geq |Z|^2$).

So we can apply Lemma \ref{minimal} to each $Q \in \Q''$ to obtain
a $(t-1)$-coloring of $Q$ which is consistent with the coloring of $G'$.
This
yields a $(t-1)$-coloring of $G$, a contradiction.
\end{proof}

\section{Refinement of the embedding}
\label{refsur}
\showlabel{refsur}

Let us assume that
an {$\alpha$-nearly embeding} in a surface $\Sigma$ is given; i.e., there
is a subset $Z\subseteq V(G)$ with $|Z|\leq \alpha$, two sets
$\mathcal V=\{(G_1,\Omega_1),\ldots,(G_{\alpha'},\Omega_{\alpha'})\}$, where $\alpha'\le\alpha$, and $\mathcal W=\{(G_{\alpha'+1},\Omega_{\alpha'+1}),\ldots,(G_{n},\Omega_{n})\}$ of societies, and
a graph $G_0$ is embedded in $\Sigma$ with Euler genus $g$.

We use the notion of radial graph. Informally, the
radial graph of an embedded graph $G$ in $\Sigma$ is the
bipartite graph $R_{G}$ obtained by selecting a point in every
region $r$ of $G$ and connecting it to every vertex of $G$ incident
to that region. However, a region maybe ``incident more than once''
with the same vertex, so one needs a more formal definition. A
\emph{radial drawing} $R_G$ is a radial graph of a 2-cell embedded
graph $G$ in $\Sigma$ if
 \begin{enumerate}
 \item $V(E(G))\cap V(E(R_G)) =V(G)\subseteq V(R_G)$;
 \item Each region $r\in R(G)$ contains a unique vertex $v_r \in V(R_G)$;
 \item $R_G$ is bipartite with a bipartition $(V(G), \{v_r \colon r\in R(G)\})$;
 \item If $e,f$ are edges of $R_G$ with the same ends $v\in V(G)$, $v_r\in V(R_G)$, then
        $e\cup f$ does not bound a closed disk in $r\cup \{v\}$;
\item $R_G$ is maximal subject to 1,2,3 and 4.

 \end{enumerate}
Finally, let $A(R_G)$ be the set of vertices, edges, and regions
  (collectively, \emph{atoms}) in the radial graph $R_G$.

By the {\it interior} of a closed walk of the radial graph we mean the union of  its elements  and the elements on the inside of the cycles it contains.
%
When we talk of a drawing of representativity
$r$ in the sphere, we are implicitly associating some three cuffs. This allows us to define
{\it interior} of a closed walk of the radial graph, which contains at most one cuff.
With this definition, the distance function defined above is also a metric in the plane \cite{RS7}.
 According to Section~9 of \cite{RS11}, we obtain the following.

\begin{theorem}
\label{metric}\showlabel{metric}
  In an embedding of representativity $r$ we can define
  a metric $d$ on $A(R_G)$ as follows:
  \begin{enumerate}
  \item If $a = b$, then $d(a,b) = 0$.
  \item If $a \neq b$, and $a$ and $b$ are  the interior
    closed walk of radial graph of length $<2r$, then $d(a,b)$ is
  half the minimum length of such a walk
  \item Otherwise, $d(a,b) = \theta$.
  \end{enumerate}
  \end{theorem}

Hereafter we refer \DEF{distance} to this metric.
Distance between two connected subgraphs $H_1, H_2$ can be defined as the smallest distance between $u \in H_1$ and $v \in H_2$. Note that $H_1$ or $H_2$ could be a single vertex $u$ or $v$.

We need the following. A similar lemma (and its proof) appears in \cite{kdh}.
\begin{lemma}
\label{refine}\showlabel{refine}
In addition to the structures in Theorem \ref{thm:3.1} with $\alpha$ and $\delta \geq 16(z+\alpha)(4\alpha^2+2z\alpha)$, where $\alpha$ is as below,
the following holds:
Let $S$ be a set of faces of order $z \geq 3$ in $G'_0$.
Then we get a refinement of the $\alpha'$-near, $\delta'$-rich embedding that satisfies the following:
\begin{enumerate}
\item
$\alpha' \leq 4\alpha^2+2\alpha z$, and $\delta' \geq \delta/2$.
\item
every vertex of $S$ is covered by a large vortex.
\end{enumerate}
Moreover, such a modification is possible in $O(n^2)$ time.
\end{lemma}

In order to show Lemma \ref{refine}, we need several lemmas.
Roughly, the following lemma says that
$S$ can be covered by a
bounded number of bounded-radius disks, where the radius is defined by the distance.
More precisely, there is a closed curve $C$ such that the graph inside the disk $D=D(C)$ (i.e., bounded by $C$)
is of bounded radius. In this case, we say that the disk $D$ is of bounded-radius (or sometimes we say that
the disk is of radius at most $l$ for some constant $l$).

Roughly, the following lemma says that
$S$ can be covered by a
bounded number of bounded-radius disks, where the radius is defined by the distance.
\begin{lemma}
\label{precenter}
In $G'_0$,
there is a set $C$ of at most $z$ \emph{centers}
such that, for each face $s$ in the set $S$, there is
a center $c$ in $C$ such that $d(c, s)\leq \alpha$.
\end{lemma}
\begin{proof}
We can greedily build the disk cover by repeatedly adding
a disk of radius $\alpha$ centered at a vertex in $s \in S$ that is not already
covered by the disks so far.  When the cover is complete, the centers of the
disks form a set $C$ such that every pair of centers has distance at least
$N_2$ (by construction).
\end{proof}

We now combine this disk cover, and make the cover disjoint,
to obtain our desired local areas of planarity as follows.
\begin{lemma}
\label{center}
In $G'_0$ there is a set $C$ of at most
$z$ vertices
such that, for each $s \in S$, there is exactly
one center $c$ in $C$ for which $d(c, s)\leq z\alpha$.
\end{lemma}
\begin{proof}
The only problem is that we might have some
of vertices in $S$
that are near to (within radius of) more than one center in $C$
(double coverage).

Suppose one disk of radius $r$ and center $c$ intersects another disk
of radius $r'$ and center $c'$.  We replace these disks by a single disk
of radius $r+r'$ and centered at a vertex of distance $r'$ from $c$
and distance $r$ from $c'$.
(Such a vertex exists by the definition of the distance)
Repeating this process, we eventually remove all intersections among disks and moreover, we can
make sure that for each $s \in S$, there is exactly
one center $c$ in $C$.

The maximum radius of any disk increases from the original maximum $N_2$
by at most a factor of the number of
disks in the original $C$, which is at most $z$.
\end{proof}

So far, we have looked at the 2-cell embedding of $G'_0$.
Next, we need to look at the $\alpha$-nearly embedding structure.
In the next lemma, we assume that $G$ has an $\alpha$-nearly embedding structure.
Note that we can easily transfer the distance in $G'_0$ to the distance in $G_0 \cup \W$.
So hereafter, when we talk about ``distance'' in $G_0$, all small vortices are also taken into account (and hence we abuse the ``distance'' in the surface, i.e., extending it to $G_0 \cup \W$).
\begin{lemma}
\label{cut structure}
Let $C$ be a bounded radius disk in $G$
such that $C$ contains exactly one vortex of depth $h$ and all vertices (in $G_0 \cup \W$) of distance  at most $r$ from the center $c\in V(G)$.
Let $G_C$ be the graph inside the disk $C$.
Assume that the representativity of $G'_0$ is at least $2r + 2h$.
Then $(G_C, V(C))$ is a vortex of depth at most $2r + 2h$.
\end{lemma}
\begin{proof}
Let us remind that $V(C)$ consists of the vertices on the boundary
of $C$.  Actually, we assume that $V(C)$ is a vertex set on the circumference
of $C$ (i.e., vertices of distance exactly $r$ from a center $c$).
We may assume that $|V(C)| > 2r$ by the definition of the distance.

Note that the graph inside $C$ is planar (up to 3-separations) with one large vortex of depth $h$,
since the representativity of $G'_0$ is at least $N_1 \geq 2r+2h$.
Let us observe that there is no separation $(A,B)$ of order at most $2r$
in $G_0 \cup \W$ such that $A$ induces a disk containing $C$ (i.e., no vertex outside $C$ is in $A$).

Take two vertices $u,v$ of $V(C)$. Since $V(C)$ is a vertex set on the circumference of $C$, $V(C)$ can be {\em partitioned} into two parts $A, B$ such that
$A$ starts $u$ and ends $v$ in the clockwise direction $C$, and $B$ starts $v$ and ends $u$ in the clockwise direction of $C$.
For any partition $A,B$ of $V(C)$ with $|A|, |B| \geq 2r$ ,
suppose there are more than $2r$ disjoint paths in the ``embedded'' graph in $C$ (i.e., the graph of $G_0 \cup \W$ inside $C$).
Then there is a vertex $x \in V(C)$
such that the curve from $c$ to $x$
hits at least $r+1$ paths of these paths.
But this does not happen
since we take $V(C)$ as a vertex set on the circumference
of $C$ (i.e., vertices of distance exactly $r$ from $c$).

Let us consider the vortex $(G_1, \Omega_1)$ of depth $h$,
contained in $C$. By Lemma~\ref{lem:omega},
for any partition $\hat A, \hat B$ of $\Omega_1$ with $|\hat A|, |\hat B| \geq 2h$,
there are at most $2h$ disjoint paths from $\hat A$ to $\hat B$ in $G_1$.
Let $(A,B)$ be the partition of $V(C)$ such
that there are more than $2r+2h$ disjoint paths from $A$ to $B$ in $G_C$.
Note that if such a partition does not exist, we know that $(G_C,V(C))$ must be a vortex of depth $2r+2h$ by Lemma~\ref{lem:omega}.

By the above argument,
there is a separation $(A',B')$ of order
at most $2r$ in the ``embedded'' graph in $C$ (i.e., the graph of $G_0 \cup \W$ inside $C$) with $A \subseteq A'$ and
$B \subseteq B'$. Note that all the vertices in $A \cap B$ are contained in $G_0$.
Actually, by using Menger's theorem,
there is  a $|A' \cap B'|$-separating face chain in $G'_0$
which separates $A$ and $B$. Here, a $k$-separating face chain is
a sequence $v_1f_1\dots,v_kf_k$ such that all $v_i$ are vertices and
all $f_i$ are faces in $G'_0$, and each face $f_i$ is incident with vertices $v_i$ and
$v_{i+1}$.

Let us observe that
if there is a large vortex $W$ of depth $h$, that is attached to a cuff $Q$,
such that $Q \cap (A'-B') \neq \emptyset$ and $Q \cap (B'-A') \neq \emptyset$,
then there is a separation $(W_1,W_2)$ of order at most $2h$ in $W$
such that $(W_1-W_2) \cap (A'-B') \neq\emptyset$ and
$(W_2-W_1) \cap (B'-A') \neq\emptyset$.

This means that if such a large vortex exists,
then there is a separation $(A' \cup W_1,B' \cup W_2)$ of
order at most $2h+2r$ in $G_C$, such that $A$ is contained in $A' \cup W_1$ and $B$ is contained in $B' \cup W_2$.
Thus by Lemma~\ref{lem:omega} we know that $(G_C,V(C))$ must be a vortex of depth $2r+2h$. This completes the proof of Lemma \ref{cut structure}.
%
\end{proof}

The same argument actually gives rise to the following.

\begin{lemma}
\label{cut structure1}
Let $C$ be a bounded radius disk in $G$
such that $C$ contains at most $k$ vortices of depth $h$ and all vertices (in $G_0 \cup \W$) of distance at most $r$ from the center $c\in V(G)$.
Let $G_C$ be the graph inside the disk $C$.
Also assume that the representativity of $G'_0$ is at least $2r + 2kh$.
Then $(G_C, V(C))$ is a vortex of depth at most $2r + 2kh$.
\end{lemma}

We also need the following result in \cite{RS12}, (7.6).

\begin{lemma}
\label{delete}
\showlabel{delete}
Let $G$ be a graph embedded into a surface of representativity $\Theta$. Let $r$ be a vertex
in $G$, and take all the vertices $W$ of distance $d$ at most $k \leq \Theta/4$.
Then the induced embedding of $G-W$ has representativity at least $\Theta-4k$, and hence we can define the new distance $d'$ in $G-W$.
Moreover,
for any two vertices $a, b \in V(G)-W$,  $d'(a,b) \geq d(a,b)-4k-2$.
\end{lemma}

We now show Lemma~\ref{refine}.
Let us assume that we have the structure in Theorem~\ref{thm:3.1} with representativity $\delta$.
Let us first focus on $G'_0$, which is obtained
from $G_0 \cup G_{\alpha+1} \cup \ldots$
by adding all the virtual edges in
$G_0 \cap G_{\alpha+i}$ for $i=1,\dots$. Then we obtain an
embedding of $G'_0$.

Take one vertex $v_i$ in $G'_0$, out of each large vortex $V_i$, and let $S'=S \cup \V''$, where $\V''=\{v_1,\dots,v_{\alpha}\}$.
By Lemma~\ref{center} there is a set $C$ of at most
$z+\alpha$ vertices
such that, for each $s\in S$, there is exactly
one center $c$ in $C$ for which $d(c, s)\leq (z+\alpha) \alpha$, and no other center $c'$ in $C$
for which $d(c', s)\leq (z+\alpha) \alpha$.
Thus we can take disjoint disks and whose center is in $C$, such that each of them is radius at most $(z+\alpha) \alpha$ from the center in $C$.
By Lemma~\ref{cut structure1}, $(G_C,V(C))$ is a vortex of depth
$2(z+\alpha) \alpha +2\alpha^2=4\alpha^2+2z\alpha$. Since there are at most $z+\alpha$
such disks, if we delete them, then by Lemma~\ref{delete}, the representativity of the resulting embedding
is $\delta - 8(z+\alpha)(4\alpha^2+2z\alpha)\geq \delta/2$.

So Lemma~\ref{refine} follows (with representativity $\delta/2$).
It follows from Lemma~\ref{lem:omega} that
there is an $O(n^2)$ time algorithm to obtain, for each $(G_C, V(C))$, a linear decomposition
of depth $4\alpha^2+2z\alpha$ (by Lemma \ref{cut structure1}). Since there are at most $\alpha+z$ large vortices (and moreover $\alpha+z$ is regarded as a fixed constant)  we obtain a desired
$O(n^2)$ algorithm as in Lemma~\ref{refine}.\qed

\section{Dealing with small vortices in a surface}
\label{smallv}

Let $G$ be a minimal counterexample to the odd Hadwiger's conjecture for the case $t \geq 6$, and
assume that $G$ has an $\alpha$-near embedding structure as in Theorem \ref{thm:3.1} with a graph
$R=K_{16k\sqrt{\log k}}$ (throughout this section, we follow the notations in Section \ref{structthm}) and $\delta \geq 10|Z|^22^{|Z|}d(3t,\alpha)$, where $d(.)$ comes from Theorem \ref{gm7} below.
The main purpose of this section is to deal with small vortices $\W$.
Let $Z=\{z_1,\dots,z_l\}$ with $l \leq \alpha$.
For each $G_i \in \W$, if there is an induced connected
bipartite graph $Q$ in $G_i$ such that $Q$ is an induced block,
%
then we take such a bipartite graph. More precisely, if we take a block decomposition of $G_i$ (i.e., take a tree-decomposition $(T,Y)$
such that for each $tt' \in T$, $|Y_t \cap Y_{t'}| =1$), then we take all $Y_t$ that are bipartite.
Note that $G_i$ may contain two or more such induced bipartite graphs.
Let $\Q$ be the union of these induced bipartite blocks in $\W$.

\begin{lemma}
\label{even1}
\showlabel{even1}
If $|\Q| > |Z|^22^{|Z|}$ then there are at least $|Z|^2$
components in $\Q$ such that all have neighbors in the same vertices $Z'$ in $Z$. Moreover $Z'$ can be reduced onto a clique of order at least $t-3$, via the odd-minor-operations, and at least $|Z|^2/2$ components in $\Q$ have
neighbors in all of the vertices in the clique and are not used to
create the clique.
\end{lemma}
\begin{proof}
Suppose $|\Q| \geq |Z|^22^{|Z|}$. Then by the pigeon hole principle, there is a set $\Q' \subset \Q$
with $|\Q'| \geq |Z|^2$ such that for any two $Q,Q' \in \Q'$, $N(Q) \cap Z =N(Q') \cap Z = Z_1 \subset Z$.

We now reduce $Z_1$ into a clique $Z_2$, via the odd-minor-operations.
We first keep identifying any two vertices $u,v \in Z_1$ by taking $Q \in \Q'$ and then taking an even path
in $(Q-G_0) \cup \{u,v\}$, if $uv$ is not present (see Fact \ref{easy3}). In this case, $(Q-G_0) \cup \{u,v\}$ is
reduced to the vertex that is obtained from the identification of $u$ and $v$, via the odd-minor-operations.

When we stuck, we still have at least $|Z|^2-|Z|$ components $Q$ remaining in $\Q'$ (see Lemma \ref{re2}).
Hence for each of these remaining components $Q \in Q'$,
we can add an edge $uv$ (with $u,v \in Z_1$) by taking an odd path in $(Q-G_0) \cup \{u,v\}$ (see Fact \ref{easy2}), if $uv$ is not present (Note that
there is no even path in $Q \cup \{u,v\}$, otherwise, we had already identified $u$ and $v$ as above).
In this case, $(Q-G_0) \cup \{u,v\}$ is
reduced to the edge $uv$, via the odd-minor-operations.

This way, we can reduce $Z_1$ onto a clique $Z_2$ via the odd-minor-operations. Let $G''$ be the resulting graph.
So all the components in $\Q'$ that are used to create this clique $Z_2$ are now gone.

In fact, for our technical purpose, we require the following: if there is a vertex $x$ in $Q-G_0$ that is adjacent to at least two vertices of $G_0$, then we first delete all edges incident with $x$ except for the ones with endpoints in $G_0$. Then we contract all these edges
into a single point. There is still above an even path or an odd path
in $(Q - G_0-\{v\}) \cup \{u,v\}$, for otherwise, there is a separation $(A,B)$ of order one in $Q$, in which case, our definition implies that
$Q$ have to be both $A$ and $B$, but not $A \cup B$.

If $|Z_2| \geq t$, we are done, as this is an odd $K_t$-model.

Let us observe that after all the above odd-minor-operations, each remaining component in $\Q'$ has neighbors to all the vertices of the clique $Z_2$, but no neighbors in $Z-Z_2$ in $G''$. Note that there are at least $|Z|^2/2$
remaining components in $\Q'$.

If $|Z_2| \leq t-4$, then for each remaining component $Q \in \Q'$, we delete $Q-G_0$, except for one ``special'' vertex that
is adjacent to all the vertices in $Z_2 \cup (Q \cap G_0)$ (if it exists).
Let $G'$ be the resulting graph of $G''$. By the minimality, $G'$ has a $(t-1)$-coloring $\sigma$.

For each $Q \in \Q'$ (including components to create a clique $Z_2$),
we now color $Q$ which extends the coloring $\sigma$.
Since $|Q \cap G_0| \leq 3$, we contract two vertices of $Q \cap G_0$ into one vertex, if there is a vertex that is adjacent to two of $Q \cap G_0$, and we leave one ``special'' vertex if it exists for the remaining component in $\Q'$,
either
\begin{itemize}
\item
at most $t-2$
colors are precolored in $Q \cup Z_2$ from $\sigma$, or
\item
exactly $t-1$ colors are precolored in $Q \cup Z_2$ from $\sigma$, but there is no vertex in $Q-G_0$ that are adjacent to
three vertices in $G_0 \cap Q$ with three different colors.
\end{itemize}

Let us first consider the first case.
If at most $t-3$ colors are used in  $Q \cup Z_2$ from $\sigma$, we can color $Q-G_0$ with two left colors.
So it remains to consider the case that exactly $t-2$ colors
are used in  $Q \cup Z_2$ from $\sigma$.
Since $Q \cap G_0$ receives at least two colors, say $a,b$, that are not used in the coloring of $Z_2$ from $\sigma$,
we can color one partite set $A_1$ of $(A_1, B_1)$ of $Q-G_0$ as $a$, and we can color $B_1$ as $x$, which is not used in
the coloring of $Q \cup Z_2$. So we obtain a $(t-1)$-coloring of each $Q$ which is consistent with the coloring $\sigma$.

Let us consider the second case. If some two vertices $x,y$ of
$Q \cap G_0$ are in a different partite set of the bipartition of $Q$,
say $x \in A_1$ and $y \in B_1$, then we color $A_1$ with $\sigma(x)$ and color $B_1$ with $\sigma(y)$ (except for the vertex $(Q \cap G_0)-\{x,y\}$ which has own color). If not, say, all the vertices $x,y,z$ of $Q \cap G_0$ are in $A_1$, then
we first pick up $\sigma(x)$ for $A_1-\{y,z\}$, and for each vertex in $B_1$, we pick up one of $\sigma(y), \sigma(z)$, which is possible because either $\sigma(y)$ and $\sigma(z)$ are same, or
there is no vertex in $B_1$ that is adjacent to both of $y,z$. So we obtain a $(t-1)$-coloring of $Q$ which is consistent with the coloring $\sigma$.

This yields
a $(t-1)$-coloring of $G$, a contradiction.

So $t-1 \geq |Z_2| \geq t-3$, and at least $|Z|^2/2$ components in $\Q'$
have neighbors to all the vertices of the clique $Z_2$, but no neighbors in $Z-Z_2$. Moreover, they are not used to create the clique $Z_2$.
\end{proof}

The above proof can handle the following case too:
\begin{quote}
when $|Z_2|=t-3$ or $t-2$, and there is no
vertex in $Q \in \Q$ that is adjacent to at least $t-3$ vertices in $Z_2$.
\end{quote}
To see this, let us first consider the case when $|Z_2|=t-3$.
For each vertex $x$ in $Q-G_0$, we have one color
that is used in $Z_2$ but that is still valid in $(Q-G_0) \cup Z_2 \cup \{x\}$.
We can plug this fact to the above proof to show that we obtain a $(t-1)$-coloring of $Q$ which is consistent with the coloring $\sigma$.
For example, if $Z_2 \cup (Q \cap G_0)$ uses at most $t-2$ colors from $\sigma$, we can
apply the same proof. On the other hand, even if $Z_2 \cup (Q \cap G_0)$ uses exactly $t-1$ colors from $\sigma$, we still have a valid color for each vertex
in $Q-G_0$. We omit further details, as they are just case-analysis.

Must of the same thing happens to the case when $|Z_2|=t-2$.
The point is that there must exist one color $x$ that is not used in
$Z_2$ in the coloring $\sigma$. So we can color one partite set with $x$
(except for $Q \cap G_0$). For the other partite set, we have two
colors available, and these two colors do not contain $x$. So we can
proceed in the same way.

Hereafter we use the following fact from this proof:
\begin{quote}
When $|Z_2|=t-3$ or $t-2$, and there is a
vertex in $Q \in \Q$ that is adjacent to at least $t-3$ vertices in $Z_2$.
\end{quote}

When $|Z_1|=t-1$, the situation is even simpler, but we deal with this case right after Lemma \ref{t52} later.
\drop{
We now assume the following.

\medskip

{\bf Assumption (*)} Suppose there are at least 5 cuffs, and moreover either
\begin{enumerate}
\item
$G'_0$ is embedded in a surface of Euler genus $g$ with representativity at least $d(3t,g)$ ($d(.)$ comes from Theorem \ref{gm7} below), or
\item
$G'_0$ is embedded into sphere.
\end{enumerate}
}

\medskip

Next, we need the following from \cite{RS7}.
A \DEF{forest} in a surface $\Sigma$ is an embedded graph in $\Sigma$ with no cycles. Two
forests $H_1,H_2$, are \DEF{homotopic} in $\Sigma$ if
\begin{enumerate}
\item[(i)]
$V(H_1) \cap bd(\Sigma) = V(H_2) \cap bd(\Sigma)$,
\item[(ii)]
for any two vertices $s, t \in V(H_1) \cap bd(H_1)$, there is a path of $H_1$ from $s$ to $t$ if and
only if there is such a path in $H_2$, and
\item[(iii)]
for any two vertices $s, t \in V(H_i) \cap bd(H_i)$, if there is a path $P_i$ of $H_i$ from $s$ to $t$ for $i=1,2$,
then $P_1$ is homotopic to $P_2$ in $\Sigma$.
\end{enumerate}
We say that forests $H_1, H_2$ in $\Sigma$ are \DEF{homoplastic} if there is a homeomorphism $\alpha : \Sigma \hookrightarrow \Sigma$ such
that
\begin{enumerate}
\item[(i)]
$\alpha(x) =x$ for all $x \in bd(\Sigma)$, and
\item[(ii)]
the forest $\alpha(H_1)$ is homotopic to $H_2$ in $\Sigma$.
\end{enumerate}

The equivalence classes of this equivalence relation are called \DEF{homoplasty classes}.
Let $G$ be a graph in $\Sigma$, and let $H$ be a forest in $\Sigma$. If there exists a forest
$H’$ homoplastic to $H$ which is a subgraph of $G$, we say that \DEF{$H$ is $G$-feasible}.

We can now state the result from \cite{RS7} (see (6.1)).
\begin{theorem}
\label{gm7}
\showlabel{gm7}
Let $G$ be a 2-connected graph on a surface $\Sigma$ of Euler genus $g$. Let $A$ be a vertex set of order $k \geq 3$ such that there are at least three faces $F$ to cover $A$. If one face $C$ contains $l$ vertices in $A$, then there is no closed curve of order at most $l-1$ whose interior contains $C$.

Then there is a function $d(|A|,g)$ that satisfies the following:
Suppose that representativity of the drawing of $G$ in a surface $\Sigma$ is at least $d(|A|,g)$ and distance between any two faces in $F$ is at least $d(|A|,g)$.
Then for any $Y$-forest with $V(G) \cap bd(\Sigma) = V(Y) =bd(\Sigma)=A$, $Y$ is $G$-feasible.
\end{theorem}

For each $G_i \in \W$, let us take an induced subgraph $R$ which contains an odd cycle and which does not contain a subgraph in $\Q$. Let $\R$ be the union of these induced subgraphs in $\W$.

We prove the following analogue of Lemma \ref{even1}.

\begin{lemma}
\label{odd1}
\showlabel{odd1}
There are no vertex disjoint subgraphs $O_1, \dots, O_l \in \R$ with $l \geq 2^{|Z|} |Z|^2$ such that
any two of $O_1, \dots, O_l$ are of distance at least $d(3t,g)+2t^2+12t$.
\end{lemma}
\begin{proof}
Suppose such $l$ disjoint subgraphs exist. By the pigeon hole principle, there is a set $\R' \subset \R$
with $|\R'| \geq |Z|^2$ such that for any two $R,R' \in \R'$, $N(R) \cap Z =N(R') \cap Z = Z' \subset Z$.
By the definition of $\R$ and Fact \ref{easy2}, for any $R \in \R'$ and
for any two vertices $u,v \in Z'$,
there is an odd path between $u$ and $v$ in $R \cup \{u,v\}$.
We can thus keep adding an edge $uv$ (with $u,v \in Z'$) from $R \in R'$ by taking an odd path in $R \cup \{u,v\}$, if $uv$ is not present. In this case, $R$ is reduced to the edge $uv$, via the odd-minor-operations. This way, we can reduce $Z'$ onto a clique, via the odd-minor-operations (see Lemma \ref{re1})\footnote{Let us observe that we could even get an even path between $u$ and $v$ in $R \cup \{u,v\}$, and as in Lemma \ref{even1},
we could identify $u$ and $v$ via this even path, if the edge $uv$ is not present.}.
Note that in these reductions, some vertices in $G_0$ may be used
(because $R$ could contain at most three vertices in a face of $G_0$).

If $|Z'| \geq t$, we are done as we obtain an odd $K_t$-model. So $|Z'| \leq t-1$.
This means that there is a vertex set $\R'' \subseteq \R'$ of order at least $3t$ such that every $R \in \R''$ is not
used to create a clique on $Z'$ (note that $t \geq 6$).

We again consider the original graph $G$.
Let $G'$ be the graph obtained from $G$ by deleting all $R - G_0$ for $R \in \R''$.
We first color $G'$. By the minimality of $G$, $G$ has a $(t-1)$-coloring $\sigma$.

As in the proof of Lemma \ref{minimal},
this coloring partitions the vertices of $Z'$ into color classes $V_1, \dots, V_l$ with $l \leq t-1$.
For each $V_i$, we can identify into a single point by taking even paths in some components in $\R'-\R''$, as above (see Fact \ref{easy3}).
On the other hand, we can also add an edge between $V_i$ and $V_j$ ($i\not=j$)
by taking an odd path in some $R' \in \R'-\R''$ as above (see Fact \ref{easy2}).
So we can reduce $Z'$ to a clique (using only components in $\R'-\R''$) such that
all $V_i$ are reduced into a single vertex. Again, note that in these reductions, some vertices in $G_0$ may be used.
We delete such vertices.
By Lemma \ref{delete}, this results in making distance smaller,
but only $4 \times t^2/2=2t^2$ for distance between any two vertices in the current graph.

Let us color $R \in \R''$. We pick up three distinct sets $L_i$ of $|Z'|$ elements in $\R''-\{R\}$.
So $|L_1|=|L_2|=|L_3|=|Z'| \leq t-1$ and $L_1, L_2, L_3$ are disjoint.
Let $R_i$ be the faces that accommodate all the sets in $L_i$ for $i=1,2,3$.
So any two faces of $R_i$ are of pairwise distance at least $d(3t,g)+12t$ in the current graph
and moreover any face in $R_i$ and any face in $R_j$ (with $i \not=j$)
are of distance at least $d(3t,g)+12t$ in the current graph.

Let $v_1, v_2, v_3$ be $R \cap G_0$ (if $|R \cap G_0| \leq 2$, we just appoint $v_1, v_2$ or just $v_1$ if $|R \cap G_0|=1$).
For a technical reason, we take a closed curve $C$ whose interior
includes $R \cap G_0$ and such that $|C|=3t-3$ and subject to that, the number of vertices interior of $C$ is as many as possible.
Let us partition the vertices on $C$ into three parts $L'_1, L'_2, L'_3$ such that all $L'_i$ are consecutive along $C$ (in a natural way).
So all the vertices on $C$ are in the clockwise order, and $L'_1, L'_2, L'_3$ partition the ordering of the vertices on $C$ into three equal size sets.
In the interior of $C$, by our choice of the closed curve $C$,
it is straightforward to see that there are three disjoint paths $P_i$ from $v_i$ to $L'_i$ (for $i=1,2,3$) such that $P_i$ does not contain
any vertex in $L'_{i-1} \cup L'_{i+1}$.
We now delete interior of $C$, and add three vertices $v'_i$ to the face bounded by $C$ such that $v'_i$ has neighbors to all of $L'_i$ (for $i=1,2,3$).
By Lemma \ref{delete}, this results in making distance smaller, but only $4 \times 3t=12t$
for distance between any two vertices in the current graph.

We now apply Theorem \ref{gm7} to the resulting graph, with respect to a forest obtained as follows:
it consists of three trees $T_1, T_2, T_3$ such that $T_i$ consists of a star with the center $v'_i$ and
each face in $R_i$ must contain a leaf. So Theorem \ref{gm7} implies that such a forest exists. Note that the distance condition is satisfied
because we only loose $2t^2+12t$ for distance between any two vertices so far.
Since $|R_i| = |Z'|$ and since all $V_i$ above are reduced into a single vertex (via the odd-minor-relations), by using the above paths $P_1, P_2, P_3$,
we can use $L_i \cup P_i$ to
reduce $\{v_1,v_2,v_3\} \cup Z'$ onto a clique, via the odd-minor-operations, such that
each node of the clique corresponds to a color class (from
the coloring $\sigma$) that is reduced into a single vertex.
Note that, from the coloring $\sigma$,  each $v_i$ may consist of a single color class or
may receive the same color as $V_j$ (for some $j$). In either case,
we can reduce $\{v_1,v_2,v_3\} \cup Z'$ onto a clique, via the odd-minor-operations.

Let $R''$ be the resulting graph of $R \cup Z'$. By the minimality,
we can color $R''$ with $(t-1)$-colors, and by our reduction of  $\{v_1,v_2,v_3\} \cup Z'$ onto a clique,
this coloring is consistent with the coloring $\sigma$.

We can do the same thing for each $R \in \R''$, and hence we obtain a $(t-1)$-coloring of $R$ that
is consistent with the coloring $\sigma$. So we obtain a $(t-1)$-coloring of the whole graph $G$, a contradiction.
\end{proof}

The last lemma in this section is crucial in our proof.
\begin{lemma}
\label{t5}
\showlabel{t5}
Suppose $Z' \subset Z$ is a clique of order at least $t-6$ (and at least $t-5$ if $g=0$) and there are $21$ vertices $S$ in $G_0$ that have neighbors to all the vertices in $Z'$ and that are pairwise far apart (i.e, pairwise distance at least $d'(g) \geq d(30,g)+180$ for some function $d'$ of $g$, to be determined later).
Then either
\begin{enumerate}
\item
$G$ has an odd $K_t$-model, or
\item
$\Sigma$ is sphere, $|Z'| \leq t-5$ and
$G-Z'$ has a near embedding in $\Sigma$ with no large vortices and with $Z-Z'=\emptyset$, or
\item
there are a set of at most twenty disks $D_1,D_2,\dots,D_l$ (with $l \leq 20$) that are bounded by closed curves $C'_1,\dots,C'_l$,
such that
if we delete all the graphs inside these disks, the resulting graph of $G'_0$ in the surface has no odd faces. Moreover, there is no graph $W'$  in $\W$ that attaches to a face (with at least two vertices) outside these disks, such that $W'$ contains an odd cycle $C$, and for some two vertices $u,v$ in $W' \cap G_0$, there are two disjoint paths from $u,v$ to $C$ in $W'$.
In addition, the graph inside each disk is a vortex of depth $40d'(g)+2\alpha^2$.
\end{enumerate}
Note that the third conclusion does not mean that the resulting graph is bipartite (if $g > 0$).
\end{lemma}
\begin{proof}
Our proof is divided into two cases:
\begin{enumerate}
\item[case (a)]
there are at least 21 odd faces that are pairwise far apart (i.e, distance at least $d'(g)$).
Here, odd faces may be obtained from a small vortex in $\W$ by adding
an even or an odd path to $G_0$.
\item[case (b)]
(a) does not happen.
\end{enumerate}
In cases (a), we shall obtain an odd $K_t$-model or the second conclusion. In case (b), we shall get the last conclusion of the lemma.

Our proof strategy is as follows: if there are 21 odd faces that are pairwise far apart (e.g., distance at least $d'(g)$), then we shall find 15 odd faces of them, and find a subset set $S'$ of $S$ of order 6 (such that each vertex in $S'$ is far from the 15 odd faces), and connect each pair of $S'$ via each odd face. If $\Sigma$ is not sphere,
this allows us to obtain an odd $K_6$-model (in $G-Z'$)
with each vertex in $S'$ in different nodes of the odd $K_6$-model. This odd $K_6$-model, together with the clique (on $Z'$) of order $t-6$, yields an odd $K_t$-minor. This corresponds to case (a).

If $\Sigma$ is sphere, and
$G-Z'$ has a near embedding in $\Sigma$ with no large vortices and with $Z-Z'=\emptyset$, we obtain the second conclusion. Otherwise,
the same argument applies and we come to case (a).

Otherwise we can adapt the argument in Section \ref{refsur} to obtain at most 20 disks that cover all
odd faces and whose graphs (inside these disks) yield a vortex of depth $40d'(g)+2\alpha^2$. Outside these disks, only even faces remain. Moreover, no small vortex that contains an odd cycle
is attached.
This gives rise to the third conclusion.

\medskip

Let us first consider case (a). Suppose there are  21 odd faces $O_1,\dots,O_{21}$ that are pairwise far apart (e.g., distance at least $d'(g)$). As mentioned above,
odd faces may be obtained by adding an even or an odd path from a small vortex in $\W$.

If $g=0$, i.e., $\Sigma$ is sphere, then we may assume that $G-Z'$ has a near embedding in $\Sigma$ with either at least one large vortex or at least one apex vertex in $Z-Z'$ (otherwise we are done).
In this case, we obtain a ``non-planar'' cross over $G'_0$. This means that either
\begin{itemize}
\item[(i)]
there is a path $P$ with endpoints $u,v$ such that $G'_0 \cup P$ is non planar, or
\item[(ii)]
there is a face $F'$ that accommodates one large vortex in $\V$ such that there are four vertices $a_1,a_2,b_1,b_2$ that appear
in this order listed, and there are two disjoint paths $P'_i$ joining $a_i$ and $b_i$ for $i=1,2$ such that they do not intersect
$G_0$ except for their endpoints. Moreover, $a_1, a_2, b_1, b_2$ are ``free'' in $G'_0$, i.e., there are four disjoint paths
from $a_1, a_2, b_1, b_2$ to any four vertices of $S$ in $G'_0$.
\end{itemize}

In the case $g > 0$, by our condition of 21 odd faces $O_1,\dots,O_{21}$ and $S$,
we can pick up 15 odd faces $O_1,\dots,O_{15}$ and a vertex set $S'$ of order 6 in $S$ such that
each vertex in $S'$ is of distance at least $d'(g)$ from the 15 odd faces. When $g=0$, we only need $|S'|=5$, and 10 odd faces $O_1,\dots,O_{10}$ that satisfy the above distance conditions, but we also need the following: in case (i),
each vertex in $S'$ is of distance at least $d'(g)$ both from $u$ and from $v$,  and each face in $O_1,\dots,O_{10}$ is of distance at least $d'(g)$ both from $u$ and from $v$. In case (ii), the face $F'$ is of distance at least $d'(g)$ both from
$S'$ and from odd faces $O_1,\dots,O_{10}$.

For each vertex $s \in S'$, we want to find paths as described above, but since degree of $s$ may be small,
this may not be possible for a trivial reason. Therefore, we take a closed curve $C$ whose interior
includes $s$ and such that $|C|=5$, and subject to that, the number of vertices in the interior of $C$ is as many as possible.
We now delete interior of $C$. We do the same thing for all the vertices in $S'$. Let $C_1,\dots,C_{6}$ be the resulting
cuffs (such that in the graph $G_0$,
each $C_i$ induces a closed curve whose interior contains $s_i \in S$). We select five vertices $v_{i,j}$ for each cuff $C_i$ for $i=1,\dots,6$ and $j=1,\dots,5$.
For each odd face $O'$
in $O_1,\dots,O_{15}$, we just take two neighbors $x,y$ of $O'$ (such that there are two independent edges between $x,y$ and $O$), and delete $V(O')$ from $G$.
Let $O'_i$ be the resulting cuff (that contains exactly two neighbors $x,y$). Let $G'$ be the resulting graph.

By Lemma \ref{delete}, this results in making distance smaller, but only $4 \times 5 \times 6 + 15 \times 4=180$
for distance between any two vertices in $G'$.
By our assumption, any two cuffs
in $C_1,\dots,C_6, O'_1,\dots,O'_{15}$ are of distance at least $d'(g) -180  \geq d(30,g)$, where $d(.)$ comes from
Theorem \ref{gm7}.

Consider first the case $g > 0$.
So by Theorem \ref{gm7}, there are 30 paths such that i) two vertices in $O'_i$ are connected
to exactly two faces in $C_1,\dots,C_{6}$, and ii) contracting $C_1,\dots,C_{6}$ into single vertices
yields a $K_6$-model. For each $s_i \in S'$,
let us take the connected graph $Q_i$ that connects $s_i$ and five of odd faces $O_1,\dots,O_{15}$,
by taking the obtained five paths with one endpoint in $C_i$, together with a connected subgraph containing $s_i,v_{i,1},\dots,v_{i,5}$.
It is straightforward to see that the latter connected subgraph can be obtained from the interior of the closed curve $C_i$ in the
original graph $G$, together with $v_{i,1},\dots,v_{i,5}$. Therefore, the obtained connected subgraphs $Q_i$ are pairwise disjoint, and any
two share one odd face in $O_1,\dots,O_{15}$. Thus by adding some part of the odd face in $O_1,\dots,O_{15}$ to
$Q_i$, we obtain disjoint connected subgraphs $Q'_1,\dots,Q'_6$ and edges $e_1,\dots,e_{15}$
such that for any two of $Q'_1,\dots,Q'_6$ there is exactly one edge of  $e_1,\dots,e_{15}$ between them, and there is a 2-coloring of $Q'_1 \cup \dots \cup Q'_6$ so that each $Q'_i$ is bichromatic, while
edges $e_1,\dots,e_{15}$ are monochromatic. Thus we obtain an odd $K_6$-model (in $G-Z'$) with each vertex in $S'$ in different nodes of the odd $K_6$-model. This odd $K_6$-model, together with the clique (on $Z'$) of order $t-6$, yields an odd $K_t$-minor. This corresponds to case (a).

Consider next the case $g=0$. Much of the same things happens. Note that in this case $|Z'| \geq t-5$, and
hence we only consider the case $|S'|=5$ (and 10 odd faces $O_1,\dots,O_{10}$. So there are at least three cuffs, which allow
us to define the metric we are using for the distance).
The only difference is that  in case (i) we need to connect
two vertices of $S'$, say $s_1, s_2$, via one odd face and $P$. So in this case, one path has to start from $s_1$ to $u$, and then we need another path from $v$ to one odd face $O$ in $O_1,\dots,O_{10}$. In addition we need one more path from $O$
to $s_2$. So in total, instead of having 30 disjoint paths in the previous case, we need 21 disjoint paths,
using $u$ and $v$. In case (ii), again, we need to connect
two vertices of $S'$, say $s_1, s_2$, via one odd face and the face $F'$. So in this case, one path has to start from $s_1$ to $a_1$, and then we need another path from $b_1$ to one odd face $O$ in $O_1,\dots,O_{10}$. In addition we need two more paths, one from $O$
to $a_2$ and the other from $b_2$ to $s_2$. So in total, instead of having 30 disjoint paths in the previous case, we need 22 disjoint paths.

If Theorem \ref{gm7} is satisfied with respect to $C_1,\dots,C_5, O'_1,\dots,O'_{10}, u, v$ for case (i) and $C_1,\dots,C_5, O'_1,\dots,O'_{10}, F'$ for case (ii), the rest of the argument is exactly the same, and again we obtain an odd $K_5$-model (in $G-Z'$) with each vertex in $S'$ in different nodes of the odd $K_5$-model. As above, this odd $K_5$-model, together with the clique (on $Z'$) of order $t-5$, yields an odd $K_t$-minor. This indeed holds for case (ii).

So it remains to consider case (i) and
the case when $u$ and $v$ are too close. To be more precise, let $l$ be distance between $u$ and $v$, if
$l \geq d(21,0)$, then we can apply Theorem \ref{gm7}, with respect to $C_1,\dots,C_5, O'_1,\dots,O'_{10}, u, v$ (and hence we are done as mentioned before).
Note that each vertex in $S'$ is of distance at least $d'(g)$ both from $u$ and from $v$, and
each face in $O_1,\dots,O_{10}$ is of distance at least $d'(g)$ both from $u$ and from $v$.

If not, we cut along the shortest curve between $u$ and $v$ (which is of length at most $l$)
to obtain the cuff $C'$ containing $u, v$,
By Lemma \ref{delete}, this results in making distance smaller, but only $4 \times l$
for distance between any two vertices in the remaining graph $G''$. So if we take $d'(g) \geq 5d(21,0)$, this implies that the resulting distance in $G''$ between
any two cuffs in $C_1,\dots,C_5, O'_1,\dots,O'_{10},C'$
is at least $d(21,0)$, and hence we can apply Theorem \ref{gm7}, with respect to $C_1,\dots,C_5, O'_1,\dots,O'_{10}, C'$ (and hence we are done as mentioned before).

So we are done if there are at least 21 odd faces that are pairwise far apart (distance at least $d'(g)$).

Suppose there are no 21 odd face that are pairwise distance at least $d'(g)$.
By Lemma \ref{refine}, there are at most twenty disks $D_1, \dots, D_l$ (with $l \leq 20$) that are bounded by closed curves $C'_1,\dots,C'_l$, such that
the graph $\hat G$ outside these disks has only even faces.
Moreover, there is no graph $W'$  in $\W$ that attaches to a face (with at least two vertices) outside these disks, such that $W'$ contains an odd cycle $C$, and for some two vertices $u,v$ in $W' \cap G_0$, there are two disjoint paths from $u,v$ to $C$ in $W'$.
In addition, each graph inside the disk $D_i$ is a vortex of depth $40d'(g)+2\alpha^2$ (because there are at most $\alpha$ large vortices of depth $\alpha$).  This corresponds to the third conclusion (and case (b)).

\end{proof}

\drop{
A similar but simpler proof for Lemma \ref{t5} works for the following lemma too. So we omit the proof.

\begin{lemma}
\label{t4}
\showlabel{t4}
Suppose $Z' \subset Z$ is a clique of order at least $t-4$
and there are $10$ vertices in $G'_0$ that have neighbors to all the vertices in $Z'$ and that are pairwise of distance at least
$d'(g)$ for some function $d'$ of $g$, to be determined later.
Then either
\begin{enumerate}
\item
$G$ has an odd $K_t$-model, or
\item
there are a set of at most five disks $D_1,D_2,\dots,D_l$ (with $l \leq 5$) that are bounded by closed curves $C'_1,\dots,C'_l$, such that
if we delete all the graphs inside these disks, the resulting graph of $G'_0$ in the surface has no odd faces. Moreover, there is no graph $W'$  in $\W$ that attaches to a face (with at least two vertices) outside these disks, such that $W'$ contains an odd cycle $C$, and for some two vertices $u,v$ in $W' \cap G_0$, there are two disjoint paths from $u,v$ to $C$ in $W'$.
In addition, each graph inside each disk is a vortex of depth $10d'(g)+2\alpha^2$.
%
\end{enumerate}
Note that the second conclusion does not mean that the resulting graph is bipartite (if $g > 0$).
\end{lemma}
}

\section{Bounding tree-width in surfaces}

This section is concerning how to bound tree-width of graphs in a surface
of $\alpha$-near embeddable graphs.
The first lemma is the following:
\begin{lemma}
\label{logn1}
\showlabel{logn1}
Let $G$ be a planar graph, and
let $C$ be a face. Then there is a subgraph $W$ of $G$ that contains $V(C)$ such that
\begin{enumerate}
\item
each vertex in $V(G-W)$ has at most three neighbors to $W$, and
\item
tree-width of $W$ is $O(\log n)$.
\end{enumerate}
\end{lemma}
\begin{proof}
We give a constructive proof. Our first phase $R_1$ is the following.
\begin{quote}
Starting with $C$, we add a vertex $v \in V(G-C)$ that has a neighbor in $C$
if the following is satisfied:

Let $C'$ be the connected subgraph of $G$ constructed so far, and let $BD(C')$ be the boundary vertices of $C'$. 
We add $v$ 
to $C'$ if $|BD(C'')| < |BD(C')|$,
where $C''$ is obtained from $C'$ by adding $v$
and all its incident edges that have another end vertex in $BD(C')$.
\end{quote}
So if $v$ has at least four neighbors to $BD(C')$, then $v$ is added.

All vertices added so far have a neighbor in $C$. Then Phase $R_1$ is done if there is no vertex that has a neighbor
in $C$ and that satisfies the above condition. Let $C_1$ be the resulting subgraph.
Every vertex in $BD(C_1)$ is of distance at most one from $C$.
Moreover, the vertices that are added in phase $R_1$ is of distance exactly one from $C$.

In the next phase $R_2$, we do the above procedure with $C$ replaced by $C_1$.
Let $C_2$ be the resulting subgraph. So every vertex in $BD(C_2)$ is of distance at most two from $C$,
and moreover, the vertices that are added in phase $R_2$ is of distance exactly two from $C$.

We continue to do the phase $R_l$ to obtain the resulting subgraph $C_l$ together with its boundary
$BD(C_l)$ whose vertices are of distance at most $l$ from $C$, and moreover,
the vertices that are added in phase $R_l$ is of distance exactly $l$ from $C$.

We claim that:
\begin{quote}
we stop at phase $R_l$ with $l \leq \log n$.
\end{quote}

To this end, for each $1 \leq l' < l$, we look at the vertices $V_{l'}$ that are added at phase $R_{l'}$.
They are of distance
exactly $l'$ from $C$. This implies that each vertex in $V_{l'}$ has no neighbors in $V_{l''}$ for
$l'' \leq l'-2$.

As observed above, adding one vertex to $V_{l'}$ contributes to reduce $BD(C_{l'})$ by at least one.
Because each vertex in $V_{l'}$ has no neighbors in $V_{l''}$ for
$l'' \leq l'-2$ and hence all its neighbors in $C_{l'-1}$ are in $V_{l'-1}$, this means $|V_{l'-1}|-|V_{l'}| \geq |V_{l'}|$ implying $|V_{l'-1}| \geq 2|V'_{l'}|$.
Since $|C| -|V_1| \geq |V_1|$, this implies that $l \leq \log |C| \leq \log n$. This proves the claim.

\medskip

So $W$ has tree-width at most $6\log n$ by the result of Eppstein \cite{epp1}, and this completes the proof.
\end{proof}

\medskip

Let us consider a graph $H$ in a surface $S$
with Euler genus $g$, and find a subgraph $W$ such that $H-W$ is planar and in addition all the vertices of $H-W$ having neighbors in $W$ are in the outer face boundary.
Such a graph $W$ is called a {\it planarizing} graph.

Recall
that a {\it noncontractible} curve $C$ in $H$ is a curve $C$ hitting
only the vertices of $H$ such that if we delete all the vertices
that hit $C$ (we shall refer to this vertex set as $V(C)$) from $H$,
then Euler genus of the resulting graph of $H$ is less than $g$. Such a
noncontractible curve is called {\it surface separating} if it
divides the surface $S$ into two regions, none of which is sphere.
Otherwise, we call it {\it surface nonseparating}. It is well-known
that there are $2g-2$ different (homology) types of surface nonseparating
noncontractible curves of $S$, and there are $g-1$ different (homology) types of surface separating noncontractible curves of $S$ (see \cite{MT}).

Let $W$ be a planarizing subgraph of $H$ that is embedded in a
surface $S$ of Euler genus $g$. By taking $W$ minimal, we may assume
that $W$ consists of at most $2g-2$ different types (i.e., the same homology class) of minimal
surface nonseparating noncontractible curves of $S$, and at most $g-1$
different types (i.e., the same homology class) of minimal surface separating noncontractible curves
in $S$, and at most $3g-3$ curves to make $W$ ``connect'' (thus there are at most $3g-3$ noncontractible curves, together with $3g-3$ curves to connect them in $W$. Note that these curves are not necessarily disjoint).
Moreover, we may assume that each of these at most $6g-6$ curves (i.e., each curve (except for the endpoints) does not hit any other noncontractible curves nor a curve joining two noncontractible curves) passes through
each face of $H$ at most once (otherwise, we can ''shorten" the
curve, see Lemma 10 in \cite{DKMW} for more details).


Let us come back to our graph $G$ that has an $\alpha$-near embedding. Let $W$ be the planariznig subgraph of $G_0$.
In our application, we have large vortices $\V$. These vortices are attached to faces $C_1,\dots,C_{\alpha'}$.
We need to connect each of these faces to $W$. This can be done as follows;
We first find a shortest curve $P_1$
between $C_1$ and $W$, and add $P_1$ to $W$ (let $W_1$ be the resulting graph), and then  find a
shortest curve $P_2$ between $W_1$ and $C_2$, and add $P_2$ to $W_1$ (let $W_2$ be the resulting graph),
and so on. Let $W'=W_{\alpha'}$ be the resulting graph after $\alpha'$ iterations. We take $W'$ so that the number of vertices in $W'$ is as small as possible.

This minimality implies the following property: for each curve $P$ (with endpoints $u,v$) in $W'$, which is either in the above at most $6g-6$ curves or joining one face $C_i$ and $W'$,
\begin{quote}[(*)]
it is the shortest (i.e, $P$ is the shortest curve between $u$ and $v$).
\end{quote}

 Let $G_l$ be the subgraph of $G'_0$ obtained from $W'$ by adding all the vertices that are of
 face-distance in $G'_0$ at most $\log n$ from $W'$ (i.e., for each vertex $u$,
 there is a vertex $v$ in $W'$ such that distance between $u$ and $v$ is at most $\log n$ in $G'_0$).

We prove the following lemma, which is crucial in our proof of our main result.

\begin{lemma}
\label{cru}
\showlabel{cru}
Tree-width of $G_l$ is at most $6400g^{5/2}(6g+2\alpha') \log n$.
\end{lemma}
\begin{proof}
By our construction of $W'$, if we cut along the boundary of $W'$, we would obtain the planar graph $G_P$ such that
the vertices on the outer face boundary are obtained from $W'$ by possibly duplicating some vertices (from cutting $W'$). Let $C$ be the resulting outer face boundary of the planar $G_P$.

By the construction of $G_l$, every vertex in $G_P$ is of face-distance at most $\log n$ from a vertex in $C'$.
So $G_P$ has tree-width at most $6\log n$ by the result of Eppstein \cite{epp1}.

We claim that even we paste all duplicating vertices of $C$ to obtain $W'$, which is embedded in the surface $\Sigma$,
tree-width of the resulting graph (that is, $G_l$) is at most $6400g^{5/2}(6g+2\alpha') \log n$, which would prove the lemma.

We need the following result of Thomassen \cite{carsten} (see Proposition 7.3.1 in \cite{MT}):

\begin{theorem}
\label{grid1} \showlabel{grid1}
Suppose $G$ is embedded in a surface of Euler genus $g$. For
any $l$, if
$G$ is of tree-width at least $400lg^{3/2}$, then it contains a flat $l$-wall.
If there is no flat
$l$-wall in $G$, then tree-width of $G$ is less than
$400lg^{3/2}$.
\end{theorem}

So by Theorem \ref{grid1}, it remains to show that there is no flat wall of height $16g(6g+2\alpha') \log n$ in $G_l$.
Suppose for a contradiction that such a flat wall $R$ exists in $G_l$.

Let us first remind that $W'$ consists of (i) at most $3g-3$ noncontractible curves, together with curves to ``connect'' them, and (ii) at most $\alpha'$
curves connecting faces $C_1,\dots,C_{\alpha'}$ (that vortices are attached to) to $W'$.
The curve in (ii) joins two vertices of $G$ while two closed curves in (i) may intersect, but
by the minimality of such closed curves, intersection of any two such curves must be at most one curve (i.e., consecutive. For if
there are two closed curves $C_1,C_2$ whose intersections consist of at least two curves, then there are a curve $P'$ in $C_1$ and $P''$ in $C_2$ such that
$P' \cup P'$ bounds a disk.
So we can
delete $P'$ to obtain two curves $C'_1, C'_2$ that are homotopic to $C_1, C_2$ respectively, but
intersections of $C'_1, C'_2$ consists only one curve. Indeed, this contradicts minimality of $W'$).

Therefore, there are at most $6g+2\alpha'$ vertices $W''$ that are contained either
in at least two curves in (i) and (ii) or in $C_1,\dots,C_{\alpha'}$.

So the flat wall $R$ contains a subwall $R'$ of height $8g  \times \log n$ such that no vertex in $W''$ is in $R'$ and in addition, distance in $R$ between any vertex in the outer face boundary $C$ of $R'$ and any vertex in $W''$ is at least $\log n$.

Let us use the fact (*). Let us take the curve $P$ as defined right before (*).
We are only interested in
the case when $|P'| \geq 10 \log n$,
where $P'$ is a curve that is obtained from $P$ by deleting  all vertices of distance at most $\log n$ from the
two endvertices $u,v$ of $P$.

Suppose there is a vertex $v'$ that is of distance exactly $\log n +1$ from  $P'$ in $R'$ and of distance at least $\log n+2$ from $C$ in $R'$.
Let $\hat W'$ be obtained from $W'$ by deleting all the vertices in $P'$.
By our choice of $R'$, $v'$ cannot be of distance within $\log n$ from
$\hat W'$ (for otherwise, $v'$ must be of distance at most $\log n$ from some curve $P''$ that is in $\hat W'$. This gives a smaller construction for $W'$, a contradiction).  Note that both $u$ and $v$ are of distance at least $\log n+2$ from $v'$.

So this implies that $v'$ must be of distance at most $\log n$ from $P'$ in $R'$, a contradiction.
Hence no such a vertex $v'$ exists.

Let us assume that $P'$ contains a vertex $v''$ that is of distance at least $4 \log n$ from $C$ in $R'$.
Since our distance is metric,
the argument in the previous paragraph implies that every vertex in $R'$ is either of distance at most $\log n$ from $P'$ or of distance $\log n$ from $C$.
We are only interested in the former case.
$v''$ divides $P'$ into two curves $P'_1, P'_2$. Let $P''_i$ be all part(s) of $P'_i$ that are of distance at most $2\log n$ from $C$ in $R'$ (note that $P''_i$ could consist of just a single curve).
Let us consider all the vertices $V_1$ ($V_2$, resp.) that are of distance at most $\log n$ from $P''_1$ ($P''_2$, resp.) in $R'$. If there is a vertex that is both in $V_1$ and $V_2$, we can make a ''short cut'' $P''$
between $P''_1$ and $P''_2$. Note that the length of curve $P''$ is at most $2\log n -2$. Since $v''$ is of distance at least $4 \log n$ from $C$ in $R'$, $P' \cup P''$ contains a curve between two endvertices of $P'$ but shorter than $P'$, a contradiction.

Hence no vertex is contained in both $V_1$ and $V_2$.
Since, again, our distance is metric, so there must exist a circumference $C_1$ ($C_2$, resp.)
for $V_1$ ($V_2$, resp.) in $R'$, i.e., distance exactly $\log n$ from $P'_i$ for $i=1,2$ in $R'$. Since every vertex in $R'$ is of distance at most $\log n$ from $P'$,
this implies that $C_{i}$ is contained in $V_{3-i}$ for $i=1,2$.
In particular, some vertex of distance at most $2\log n$ from $C$ which is contained in $C_2$ is also contained in $V_1$.
But again,  as above, we can make a ''short cut'' $P''$
between $P''_1$ and $P''_2$ by taking a shorter curve between two endvertices of $P'$ in $P' \cup P''$. Note that the length of curve $P''$ is at most $2\log n$. This completes the proof.
\end{proof}

\section{Finale}
\label{fina}
\showlabel{fina}

In this section, we shall finish the proof of Theorem \ref{main1}. In order to do so, we must have a closer look at
the second and
the third conclusions of Lemma \ref{t5}. For these two lemmas, $G$ is a minimal counterexample to the odd Hadwiger's conjecture
for the case $t \geq 6$.

The following lemma is concerning the second conclusion of Lemma \ref{t5}.
\begin{lemma}
\label{t51}
\showlabel{t51}
Assume that there is a vertex set $Z'$ of order at most $t-5$ in $G$ such that $G-Z'$ has a near embedding in sphere with no large vortices, and $Z'$ induces
a clique.
Suppose that
\begin{enumerate}
\item
there are ten vertices $S$ in $G_0$ that have neighbors to all the vertices in $Z'$ and that are pairwise of distance
$d(24,0)+48$, where $d(,)$ comes from Theorem \ref{gm7}, and
\item
there are seven faces $F_1, \dots, F_7$ in $G_0$ that are pairwise of distance
$d(24,0)+48$, and that satisfy the following property:
for each $i$, either
\begin{itemize}
\item
$|F_i|$ is odd, or
\item
there is a graph $W'$  in $\W$ that attaches to $F_i$ with at least two vertices, such that $W'$ contains an odd cycle $C$,
and for some two vertices $u,v$ in $W' \cap G_0$, there are two disjoint paths from $u,v$ to $C$ in $W'$.
\end{itemize}
\end{enumerate}
Then $\W=\emptyset$.
\end{lemma}
\begin{proof}
By Lemma \ref{conn1}, there is no graph in $\W$ that is bipartite (because $|Z'| \leq t-5$). Suppose that there is a face $C$ that accommodates a non-bipartite graph  $W \in W$ in $G'_0$.
By minimality, we can $(t-1)$-color $G-(W-(W \cap G_0))$. This yields a coloring of $Z'$ and $W \cap G_0$.
So we can partition $Z' \cup (W \cap G_0)$ into color classes $V_1,\dots,V_l$, where $l \geq |Z'|$ because $Z'$ is a clique.
Let us assume that $W \cap G_0$ are contained in $V_1 \cup V_2 \cup V_3$. We shall, in $G-(W-(W \cap G_0))$,
identify each of $V_1, V_2, V_3$ into a single point, via the odd-minor-operations,
such that there is an edge between any two $V_i,V_j$, as in the proof of Lemma \ref{odd1} (note that $Z'$ is a clique of order at most $t-5$). This allows us to $(t-1)$-color $W$ whose coloring is consistent
with the coloring of $G-(W-(W \cap G_0))$ (and hence we obtain a $(t-1)$-coloring of the whole graph, a contradiction).
It remains to show that such a reduction exists.

By our assumptions, we may assume that
\begin{itemize}
\item
there are three vertices $S'=\{v_1,v_2,v_3\}$ in $S$ that are of distance at least $d(24,0)+48$ from $W$,
\item
there are six faces (out of the seven faces $F_1, \dots, F_7$), say, $F_1, \dots, F_6$, that are of distance at least $d(24,0)+48$ from $W$, and
\item
distance between any vertex in $S'$ and any face in $F_1, \dots, F_6$ is at least $d(24,0)+48$.
\end{itemize}

As in the proof of Lemma \ref{t5}, for each face $O$ of $F_1, \dots, F_6$, we take three neighbors $u,v,w$ of $O$ (such that there are three independent edges between $u,v,w$ and $O$) and delete $F_1, \dots, F_6$ from $G'_0$.
Let $F'_i$ be the resulting cuff (containing exactly three vertices $u,v,w$).
Again, as in the proof of Lemma \ref{t5}, by Lemma \ref{delete},
this results in making distance smaller, but only $4 \times 2 \times 6=48$ for distance
between any two vertices in the remaining graph $G'$ of $G'_0$.


We now consider the following twelve disjoint paths $P_1,\dots,P_6,P'_1,\dots,P'_6$ in $G'$: connect $v_i$ and $F'_i$ for $i=1,2,3$ to obtain paths $P_i$, connect
$F'_i$ and $F'_{i+3}$ for $i=1,2,3$ to obtain paths $P'_i$, connect
$(F'_4, F'_5)$, $(F'_5, F'_6)$, $(F'_6, F'_4)$ to obtain paths $P'_4, P'_5, P'_6$, and connect $F'_i$ and $C \cap W$ to obtain paths $P_{i+3}$ for $i=1,2,3$.
By Theorem \ref{gm7}, such twelve disjoint paths must exist.

Now it is straightforward to see that $P_i \cup P_{i+3} \cup F_i$, together with $P'_1, \dots, P'_6$,
leads to identify each of $V_1, V_2, V_3$ into a single point,
via the odd-minor-relations,
such that there is an edge between any two $V_i,V_j$ for $i,j \leq 3$ (by the paths $P'_1, \dots, P'_6$). Moreover, each of $V_1, V_2, V_3$ has an edge to $V_j$ for $j \geq 4$ because $P_i \cup P_{i+3} \cup F_i$ contains one vertex in $S$. This completes the proof.
\end{proof}

The following lemma is concerning the third conclusion of Lemma \ref{t5}.
\begin{lemma}
\label{t52}
\showlabel{t52}
Assume that there is a vertex set $Z'$ of order at least $t-3$ such that $Z'$ induces a clique, and
with $Z' \subset Z$, $G-Z$ has a nearly embedding in a surface $\Sigma$ of Euler genus $g$.

Suppose furthermore that
there is a set of $l \geq 1$ disks $D_1,D_2,\dots,D_l$ that are bounded by closed curves $C'_1,\dots,C'_l$, such that
if we delete all the graphs inside these disks, then
\begin{itemize}
\item
the resulting graph of $G'_0$ in the surface has no odd faces and has representativity at least $d(6,g)$, if $\Sigma$ is not sphere, and
\item
there is no graph $W'$ in $\W$ that attaches to a face in the resulting graph of $G_0$, such that $W'$ contains an odd cycle $C$, and for any some  vertices $u,v$ in $W' \cap G_0$, there are two disjoint paths from $u,v$ to $C$ in $W'$.
\end{itemize}
Then there are no five vertices in $G_0$ nor five components in $\Q$ that are pairwise of distance at least $d(6,g)$, that are of distance at least $d(6,g)$ from any of the  disks $D_1,D_2,\dots,D_l$, and
that have neighbors to all the vertices in $Z'$. Note that $d(,)$ comes from Theorem \ref{gm7}.
\end{lemma}
\begin{proof}
Suppose there are five vertices $\{s_1,s_2,s_3,s_4,s_5\}=S$ in $G_0$ that are of distance at least $d(6,g)$ from any of the  disks $D_1,D_2,\dots,D_l$, and
that have neighbors to all the vertices in $Z'$. We only consider the case when $|Z'|=t-3$, as other cases easily follow from this
most difficult case. Moreover, the proof for the case for the five components in $\Q$ is exactly the same. Indeed, as remarked right
after Lemma \ref{even1}, there is one vertex in each component of
$\Q$ that is adjacent to all the vertices of $Z'$. We can think of this
vertex as $S$, and the same argument below can apply\footnote{Strictly speaking, if $|Z'|=t-2$, we need to say ``ten components'' (instead of five as in
Lemma \ref{even1}) in the conclusion. But this does not make difference, so we do not bother to change it. We do not deal with the case $|Z'|=t-1$ here. We refer the reader to the remark right after the proof of Lemma \ref{t52} for the case when $|Z'|=t-1$.}.
So we only consider $S$ in $G_0$. Let us consider one face $F$ of odd size in $G_0$ (so this face is in one of the disks).

Note that the distance $d(,)$ can be only defined for either the sphere with at least three cuffs or the embedding in a surface (with positive Euler genus) with representativity at least $d(6,g)$. So our assumption implies that we assume either of these two situations.

Let $\hat G$ be obtained from $G_0$ by deleting all the graphs inside the disks $D_1,D_2,\dots,D_l$ that are bounded by closed curves $C'_1,\dots,C'_l$.
So every face in $\hat G$ is of even size.

Suppose $\hat G$ is bipartite. Then the idea is to first pick up three vertices $S' \subset S$ such that $S'$ are contained in the same partite set of the bipartite graph $\hat G$, and then find three disjoint paths from $S'$ to $F$. In this way, we obtain an odd $K_3$-model (in $G-Z'$) with the three vertices (in $S'$) in three different nodes of
the odd $K_3$-model. This odd $K_3$-model, together with $Z'$, gives rise to an odd $K_t$-model because every vertex in $S'$ is
adjacent to all the vertices in $Z'$.

This can be achieved by first
taking some three vertices $a_1, a_2, a_3$ in $F$. Then if there are three disjoint paths from $S'$ to $a_1,a_2, a_3$
that are internally disjoint from $F$ (except for their end vertices), we are done. Indeed, such three paths exist
by Theorem \ref{gm7} (by choosing appropriate three vertices $a_1, a_2, a_3$).

It remains to consider the case when $\hat G$ is not bipartite. In this case, we need some notations.

We are now given a surface $\Sigma$ with disks $D_1,\dots, D_l$, and
we consider an embedding $\sigma$ of $G_0$ induced by $\hat G$
such that every face is of even size. Two curves $P_1, P_2$ with the same endvertices in $\sigma$ are called \DEF{homoplastic} if there is a homeomorphism $\alpha : \Sigma \hookrightarrow \Sigma$ such
that
\begin{enumerate}
\item[(i)]
$\alpha(x) =x$ for all $x \in bd(\Sigma)$, and
\item[(ii)]
the curve $\alpha(P_1)$ is homotopic to $P_2$ in $\Sigma$.
\end{enumerate}

The equivalence classes of this equivalence relation are called \DEF{homoplasty classes}.
We can also define \DEF{homoplastic} and \DEF{homoplasty classes} for closed curves in $\Sigma$.

For any two closed curves $P_1, P_2$ with the same endvertices in $\sigma$, if they are homotopic,
then parity of $P_1$ is the same as that of $P_2$, because $P_1 \cup P_2$ bounds a disk and the graph inside
this disk is an induced bipartite graph (since each face is of even size). We claim that
the same conclusion holds even if $P_1$ and $P_2$ are homoplastic. Again, $P_1 \cup P_2$
bounds a chain of disks from one endvertex to the other endvertex, and graphs inside these disks
are induced bipartite, so parity of $P_1$ is the same as that of $P_2$. The same argument
can be applied to the case when two closed curves are either homotopic or homoplastic.

Let us consider the graph inside the disk $D_1$.
It contains the odd face $F$.
We take three disjoint paths from $F$ to $V(D_1)$ such that
they do not intersect $F$ except for the endpoints. It is straightforward to see that such paths exist in the graph inside the disk $D_1$.
Let $v_1, v_2, v_3$ be the endpoints of these three paths in $V(D_1)$.

Now pick up five faces $F_1, F_2, F_3, F_4, F_5$ such that $s_i$ is in $F_i$ for $i=1,2,3,4,5$.
By the above observation, we know that disks $D_1,\dots,D_l$ and faces $F_1, F_2, F_3, F_4, F_5$ define homoplasty classes
with respect to paths with endvertices in $D_1 \cup \dots \cup D_l \cup F_1 \cup F_2 \cup F_3 \cup F_4 \cup F_5$.
We first figure out the parity of paths from $s_i$ to $v_1, v_2, v_3$.
By pigeon whole principle, we may assume that
we can specify three disjoint paths $P_i$ from $s_i$ to $v_i$ for $i=1,2,3$, with the same parity (by choosing
appropriate homoplasty class, because $\hat G$ is not bipartite), and by Theorem \ref{gm7},
we can find such three disjoint paths $P_1, P_2, P_3$ with this specified parity. Because the graph inside the disk $D_1$ contains the odd face $F$ and the above three disjoint paths from $F$ to $V(D_1)$, three disjoint paths $P_1, P_2, P_3$ allow us
to obtain an odd $K_3$-model (in $G-Z'$) with the three vertices ($v_1, v_2, v_3$) in three different nodes of
the odd $K_3$-model. This odd $K_3$-model, together with $Z'$, gives rise to an odd $K_t$-model because every vertex in $S$ is
adjacent to all the vertices in $Z'$.
\end{proof}

Let us observe that the same proof still holds in Lemma \ref{t52} even
if $l=0$, but $G-Z'$ is not bipartite. We just need the three
disjoint paths so that we obtain an odd $K_3$-model (in $G-Z'$) with some three vertices (in $S$) in three different nodes of
the odd $K_3$-model. This is exactly the same as the above proof,
so we omit the details.

We also remark that if $|Z'|=t-1$, then we can show that either there is an odd $K_t$-model or the conclusion of Lemma \ref{t52} holds.  To see this, let us first remark that there is no vertex in $G_0$ that is adjacent to all the vertices in $Z'$. Let us also note that if one partite set of $Q$ sees all $t-1$ vertices in $Z'$, we can contract this partite set into one to obtain an odd $K_t$-minor. Otherwise, we can show that either $A_1$ or $B_1$ sees at least $t-3$ vertices in $Z_1$, by following the remark right after Lemma \ref{even1} (since we only need case-analysis, we omit details).

So one component $Q_1$ of $\Q$ can be contracted into a single point $v'$ so that it sees all but at most two vertices (say $a,b$) in $Z_1$. Take two more components $Q_2, Q_3 \in \Q$.
We try to connect $v'$ to both $Q_2$ and $Q_3$ so that we can obtain edges $v'a,v'b$, via odd-minor-operations.
To do this, we first apply Lemma \ref{t5} to confirm that the assumption of Lemma \ref{t52} is satisfied with $l \leq 3$.
We then apply the first half of the
above proof of Lemma \ref{t52} to $Q_1, Q_2, Q_3$ (using either an odd face or odd non-contractible cycles in the surface) to obtain an odd $K_t$-minor.
Since the proof is identical to the first half of that given in Lemma \ref{t52}, we omit the proof.

\medskip

We are now ready to prove Theorem \ref{main1}.\par

\medskip

{\bf Proof of Theorem \ref{main1}.}\par

\medskip

Let us first point out that we may assume that
a given graph $G$ is a minimal counterexample to the odd Hadwiger's conjecture
for the case $t \geq 6$. For otherwise, we can perform a reduction. Then we apply the whole argument to the resulting graph.
So we now assume that there is no more reduction for $G$.

We first apply Theorem \ref{thm:extended} to $G$
with $k=t$ and $\delta=100000 \alpha^{6} 2^{2\alpha} \times d(3t,g)$. Note that if 1 or 2 of Theorem \ref{main1} happens, we are done.

So we obtain a tree-decomposition $(T,Y)$, as in Theorem \ref{thm:extended}. If tree-width of $G$ is at most $100000 \alpha^{5} 2^{\alpha} \times d(3t,g) \log n$, we are done.

There are two cases:\par

\medskip

\DEF{Case 1.} There is a bag $Y_t$ which is nearly bipartite (i.e., the first conclusion  in Theorem \ref{thm:extended}).\par

\medskip

Apply Lemma \ref{re4} to confirm that the number of children bags of $Y_t$ is at most $\alpha^2 2^{\alpha}$. This indeed confirms the degree condition
on $Y_t$ in the second conclusion of 3 in Theorem \ref{main1}.\par

\medskip

\DEF{Case 2.} There is a  a bag $Y_t$ which is $\alpha$-near embedded (i.e., the second conclusion in Theorem \ref{thm:extended}).\par

\medskip

For simplicity, we use $G_0$ for the surface part of an $\alpha$-near embedding of $Y_t$. We also follow the notations for $\alpha$-near embeddings ($\W, \V$ etc).

For each $G_i \in \W$, if there is an induced connected
bipartite graph $Q$ in $G_i$ such that $Q$ is an induced block,
%
then we take such a bipartite graph. More precisely, if we take a block decomposition of $G_i$ (i.e., take a tree-decomposition $(T,Y)$
such that for each $tt' \in T$, $|Y_t \cap Y_{t'}| =1$), then we take all $Y_t$ that are bipartite.
Note that $G_i$ may contain two or more such induced bipartite graphs.
Let $\Q$ be the union of these induced bipartite blocks in $\W$.

\drop{
For each $G_i \in \W$, if there is an induced connected bipartite
graph $Q$ in $G_i$ such that the separation $(Q,G_i-Q)$
has order at most one or $Q = G_i$, then we take such
a bipartite graph. Note that $G_i$ may contain two or more such induced bipartite graphs. Let $\Q$ be the
union of these induced bipartite components in $\W$.}

For each $G_i \in \W$, let us also take an induced subgraph $R$ which contains an odd cycle and which does not
contain a subgraph in $\Q$. Let $\R$ be the union of these induced subgraphs in $\W$.
Apply Lemma \ref{odd1} to obtain the conclusion as in Lemma \ref{odd1}, which, by Lemma \ref{refine}, implies that there are at most
$l \leq 2^{\alpha}\alpha^2$ disks $\D'=\{D_1, \dots, D_l\}$ (that are bounded by closed curves $C'_1,\dots,C'_l$)
that cover all the components in $\R$,
such that each graph in a disk in $\D'$ is a vortex of depth
$2^{\alpha+1}\alpha^2 \times d(3t,g)$.

Let us observe that each component $Q$ in $\Q$ has at least $t-2$ neighbors to $Z$.
For otherwise, take a vertex $v$ in $Q-G_0$, and delete all edges incident with $v$, except for the ones in $E(Q)$. Contract these edges into a single point. Let $G'$ be the resulting graph. By minimality, $G'$ has a  $(t-1)$-coloring which
can be extended to $Q$.

\drop{
We then keep identifying two vertices $u, v \in Z$
(if $uv \not \in E(G)$) using components in $\Q$
(see Fact \ref{easy3}). When we stuck, let $\Q'$ be the components
in $\Q$ that are used for this identification. Let $Z_1$ be the resulting
set of $Z$.
By Lemma \ref{even1},
if $|\Q| > |Z|^22^{|Z|}$, then $Z$ contains a clique $Z_1$
of order at least $t-3$, and moreover there are at least $|Z|^2/2$
elements in $\Q-\Q'$ that have neighbors to all of the vertices in
the clique in $Z_1$.
}


Let $U$ be a vertex set in $G_0$ such that each vertex in $U$ has at least $t-6$ neighbors
in $Z$ (at least $t-5$ neighbors in $Z$ when $\Sigma $ is sphere).
\drop{
As above, we keep identifying two vertices $u, v \in Z_1$
(if $uv \not \in E(Z_1)$) using vertices in $U$
(see Fact \ref{easy3}). When we stuck, let $\hat U \subset U$ be vertices that are used for this identification. Let $Z_2$ be the resulting
set of $Z_1$.


At the moment, if there are two non adjacent vertices $u,v$ in $Z_2$,
then this means that there is no component in $\Q-\Q'$ nor a vertex in $U-\hat U$
that has neighbors to both $u$ and $v$.
}

We need to consider two cases:\par

\medskip

\DEF{Case 2.1.} $|U|+ |\Q| > 2\alpha^22^{\alpha}$, and either $U$ or $\Q$ contains twenty one elements that are pairwise of distance $d(30,g)+180$, where $d(,)$ comes from Theorem \ref{gm7}. Note that $d'(g) \geq d(30,g)+180$ in Lemma \ref{t5}. 

\medskip

In this case, if $|\Q| \geq \alpha^22^{\alpha}$, then
by Lemma \ref{even1},
we can reduce a subset of $Z$ to a clique $Z_1$
of order at least $t-3$, via odd-minor-operations.
Let $\Q' \subset \Q$ be the components that are used to create
$Z_1$ via odd-minor-operations. By Lemma \ref{even1},
there are at least $\alpha^2/2$
elements in $\Q-\Q'$ that have neighbors to all of the vertices in
the clique in $Z_1$ (note that no neighbors of each component in $\Q'$ is in $Z-Z_1$). Let them be $\Q''$.

By the same argument, if $|U| \geq \alpha^22^{\alpha}$,
we can still reduce subset of $Z$ to a clique $Z_1$
of order at most $t-1$, via odd-minor-operations.
Let $\hat U \subset U$ be the vertices that are used to create
$Z_1$ via odd-minor-operations. Again, by the same argument, there
are at least $\alpha^2/2$ vertices in $U-\hat U$ that have neighbors to all of the vertices in
the clique in $Z_1$. Let them be $\hat U'$.

Following, we assume one of the following happens:
\begin{itemize}
\item[(i)]
$Z_1$ is obtained from $\Q'$ (so $|Z_1| \geq t-3$ by Lemma \ref{even1}), and there are twenty one elements in $\Q''$ that are pairwise of distance $d(30,g)+180$, or
\item[(ii)]
$Z_1$ is obtained from $\hat U$, and $|Z_1| \geq t-6$ (or $|Z_1| \geq t-5$ when $\Sigma$ is sphere). Moreover, there are twenty one vertices in
$\hat U'$ that are pairwise of distance $d(30,g)+180$.
\end{itemize}

In case (ii),
apply Lemma \ref{t5} with all the cliques  that satisfy $Z_1$. Since $|Z| \leq \alpha$, there are at most $\alpha^t$ such cliques.
Similarly, in case (i), we shall apply Lemma \ref{t5} with all the cliques  that satisfy $Z_1$. Note that by Lemma \ref{even1}, 
$|Z| \geq t-3$. 
Let us give more details for this case. If $|Z_1|=t-2$ or $t-3$, then as remarked right after the proof of Lemma \ref{even1},
there is a vertex $v$ in $Q$ that is adjacent to at least $t-3$ neighbors of $Z_1$. So we can apply Lemma \ref{t5} with
these vertices that have at least $t-3$ neighbors in $Z_1$.

Note that our condition on $\delta$ guarantees that we can apply Lemma \ref{t5}.
If we obtain the third conclusion for some cliques,
we apply Lemma \ref{t52}. Note that if the first conclusion happens for some clique, we are done.

Suppose the second conclusion of Lemma \ref{t5} holds for some clique. Note that $\V=\emptyset$.
From Lemmas \ref{t51} and \ref{refine},
then the following holds: either
\begin{enumerate}
\item
$\W=\emptyset$ and $|Z_1| \leq t-5$, or
\item
there are at most six disks $\D_1$ whose interior are of radius at most $6(d(24,0)+24)$, such that no odd face in $G_0$ is outside these disks. Moreover, there is no graph $W'$  in $\W$ that attaches to a face (with at least two vertices) outside these disks, such that $W'$ contains an odd cycle $C$, and for some two vertices $u,v$ in $W' \cap G_0$, there are two disjoint paths from $u,v$ to $C$ in $W'$.
\end{enumerate}

Note that our condition on $\delta$ guarantees that we can apply Lemma \ref{t51}.
In the former case, we first $(t-5)$-color $Z_1$. This gives rise to
a coloring of $Z$. We can then
4-color $G-Z$ by the Four Color Theorem, because
$G-Z$ is planar, and the fact $|Z_1| \leq t-5$ implies that $Z$ only uses $(t-5)$-colors (so $G$ is $(t-1)$-colorable, and we are done). So if there is one clique $Z_1$ that satisfies the former
case, we are done. So we may assume that this would not happen for
all cliques that satisfy $Z_1$ as in (ii).

\drop{
So consider the second case. Apply Theorem \ref{refine} with disks $\D_1$. Let $W$ be a graph obtained from Theorem \ref{refine}.
For each vertex $x$ in $G_0-W$ that has at least two neighbors,
contract all neighbors (which are independent) into one.
This is possible because there is no odd face outside the disks.

Let $W'$ be the resulting graph. By the minimality of $G$,
$W' \cup Z$ has a $(t-1)$-coloring. Let us uncontract each $x$.
Now the coloring of $W'$ also gives rise to a $(t-1)$-coloring of $W$
such that each $x$ sees at most one in the boundary of $W$.
This means that every vertex outside $W$ but having a neighbor in $W$
has two colors available (because it has at most $t-4$ vertices in $Z$). Since $G-W$ is planar, each vertex in the outer face boundary of $G-W$ (which may have neighbors to $W$) has
two colors available, and every vertex in $G-W$ that is not in the
outer face boundary of $G-W$ has three colors available (because it
does not have a neighbor in $W$).

Note that $G-Z-W'$ is bipartite, there is a valid coloring for $G-Z-W'$
that extends the coloring of $W \cup Z$.
So we obtain a $(t-1)$-coloring of $G$, a contradiction.
}

\medskip

Suppose the third conclusion of Lemma \ref{t5} holds for some cliques.
By Lemma \ref{even1}, if we apply case (i), then $|Z_1| \geq t-3$. But on the other hand, if we apply case (ii),
then $|Z_1|$ may be $t-4$ or $t-5$ or $t-6$.
Below, we are applying Lemma \ref{t52} (Again, our condition on $\delta$ guarantees that we can apply Lemma \ref{t52}).
Therefore, if we apply case (ii), then let $\hat U_1 \subset \hat U$ be a vertex set
that creates the clique $Z_1$ of order at least $t-3$, and let $\hat U_1' \subset \hat U'$ be a vertex set that have neighbors to all of the vertices in the clique $Z_1$. 

We only deal with the case when $|Z_1|=t-2$ or $t-3$.
Note that as remarked right after in the proof of Lemma \ref{even1},
if $|Z_1|=t-2$ or $t-3$, there is a vertex in each component of $\Q' \cup \Q''$ that is adjacent to at least $t-3$  vertices in $Z_1$.
If $|Z_1|=t-1$, as remarked right after Lemma \ref{t52},
we have at most two disks that cover all the components in $\Q' \cup \Q''$.
So we only focus on the cases $|Z_1| \leq t-2$.

From Lemmas \ref{t5},  \ref{t52} and \ref{refine}, the following holds:
there is a set of $l \leq 20$ disks $\D_2=\{D_1,D_2,\dots,D_l\}$ such that, graphs inside these disks yield a vortex of depth $40(d(30,g)+180)+2\alpha^2$ (these constants come from Lemma \ref{t5}), and if we delete all the graphs inside these disks, then
\begin{itemize}
\item
the resulting graph in the surface has no odd faces, and
\item
there is no graph $W'$  in $\W$ that attaches to a face in the resulting graph, such that $W'$ contains an odd cycle $C$, and for some two vertices $u,v$ in $W' \cap G_0$, there are two disjoint paths from $u,v$ to $C$ in $W'$.
\end{itemize}

Moreover, for each clique $Z_1$, either
\begin{enumerate}
\item[(a)]
there are another four disks $\D_3$ whose interior are of radius at most $4d(6,g)$, that cover all the vertices in all components in $\Q''$, or
\item[(b)]
there are another four disks $\D_3$ whose interior are of radius at most $4d(6,g)$, that cover all the vertices in $\hat U_1'$.
\end{enumerate}


As mentioned above, we apply the above arguments for all the cliques $Z_1$ as in (i) and (ii). We may also interchange $\Q'$ and $\Q''$ (also $\hat U$ and $\hat U'$); namely we construct the clique $Z_1$ via odd-minor-operation from components in $\Q''$.
Let $\D_1', \D_2', \D_3'$ be the unions of $\D_1, \D_2, \D_3$, respectively.

\medskip

So in summary, by applying Lemma \ref{refine} if some disks in $\D'_1 \cup \D'_2 \cup \D'$ have overlaps, the following holds:
there is a set of $l \leq 1000 \alpha^{t+2} 2^{\alpha}$ disks $\D'_1 \cup \D'_2 \cup \D'=\{D_1,D_2,\dots,D_l\}$ such that, graphs inside these disks yield a vortex of depth $10000 \alpha^{t+5} 2^{2\alpha} \times d(3t,g)$, and if we delete all the graphs inside these disks, then
\begin{itemize}
\item
the resulting graph in the surface has no odd faces, and
\item
there is no graph $W'$  in $\W$ that attaches to a face in the resulting graph, such that $W'$ contains an odd cycle $C$, and for some two vertices $u,v$ in $W' \cap G_0$, there are two disjoint paths from $u,v$ to $C$ in $W'$.
\end{itemize}

In addition,
from (a) and Lemmas \ref{even1} and \ref{t52}, we obtain the following:
\begin{quote}
There are at most $\alpha^22^{\alpha}$ components in $\Q$ that
attach to $G_0 - (\D'_1 \cup \D'_2 \cup \D')$. Let $\D''$ be the disks that cover all components in $\Q$ in $G_0-\D'_1-\D'_2-\D'$. 
So each disk in $\D''$ is of radius one. 
\end{quote}


By our choice, for each vertex $x$ in $G_0 - (\D'_1 \cup \D'_2 \cup \D' \cup \D'')$, its neighbors in $G-Z$ form an independent set.
So $x$ has to have at least $t-2$ neighbors in $Z$ (For otherwise, just delete all edges from $x$ to $Z$, and contract all other edges incident with $x$. Let $G'$ be the resulting graph. By minimality, $G'$ has a $(t-1)$-coloring, and this coloring can be easily
extended to $x$, a contradiction). If there are more than $\alpha^22^{\alpha}$ such vertices $T$, then  as argued before, we can obtain
a clique $Z_2$ in $Z$ via odd-minor-operations (only using these vertices). By (b), we know that $|Z_2| \leq t-4$.
Let $\hat T \subset T$ be the vertices to create this clique $Z_2$. For each vertex $t \in \hat T$,
we pick up two neighbors $t_1,t_2$ in $Z$ that would be contracted into a single vertex (to create the clique $Z_2$).
We then delete all edges whose endpoints are $t$ and a vertex in $Z$, except for $tt_1, tt_2$. Then contract remaining edges
incident with $t$ into a single vertex. Let $G'$ be the resulting graph.
Again by minimality, $G'$ has a $(t-1)$-coloring, and this coloring can be easily
extended to each $t$, a contradiction. Thus it follows that
\begin{quote}
there are at most $\alpha^22^{\alpha}$ vertices in $G_0 - (\D'_1 \cup \D'_2 \cup \D' \cup \D'' \cup D_V)$, 
\end{quote}
where $\D_V$ consists of the  disks that accommodates vortices in $\V$ .
So each disk in $\D_V$ is of radius one, and $|\D_V| \leq \alpha$.

Moreover $G$ consists of the followings:
\begin{quote}
$Z$, at most $\alpha^22^{\alpha}$ vertices in $G_0 - (\D'_1 \cup \D'_2 \cup \D' \cup \D'')$ and graphs inside the disks $\D_1' \cup D_2' \cup \D' \cup \D'' \cup \D_V$.
\end{quote}

We shall later prove that such a graph is of tree-width at most $c \log n$ for some constant $c$ that only depends on $t$.
Note that there may be two disks in $\D'_1 \cup \D'_2 \cup \D' \cup \D'' \cup D_V$, but by Lemma \ref{refine}, the number of disjoint disks in $\D'_1 \cup \D'_2 \cup \D' \cup \D'' \cup D_V$ is at most
$10000 \alpha^{t+2} 2^{\alpha}$, and
graphs inside these disks yield a vortex of depth
of depth at most $100000 \alpha^{t+5} 2^{2\alpha} \times d(3t,g)$.

\medskip

\DEF{Case 2.2.} $|U| + |\Q| \leq  2|Z|^22^{|Z|}$ or
both $U$ or $\Q$ does not contain twenty one elements that are pairwise of distance $d(30,g)+180$, where $d(,)$ comes from Theorem \ref{gm7}.

\medskip

Let $\D_5$ be the faces of $G_0$ that cover
all the vertices in $U$ and all the elements of $\Q$.
By the assumption of Case 2.2,
$|\D_5| \leq 2|Z|^22^{|Z|}$.

Apply Lemma \ref{cru} to $G_0$, with
$\D', \D_5, \D_V$ (as in Case 2.1.). Let $\hat W$ be a graph obtained from Lemma \ref{cru}.
Note that, by Lemma \ref{refine}, the number of disks is at most
$10000 \alpha^2 2^{\alpha}$, and each disk can accommodate a vortex
of depth at most $100000 \alpha^5 2^{2\alpha} \times d(3t,g)$.

Let us observe that, at this moment,
there are no graphs in $\W$ outside
the disks. So $G_0-Z_2-\hat W$ is planar.

We now assume the following:
\begin{quote}
$G_0- \hat W$  is not empty.
\end{quote}

Again,
in this case, we are trying to make a reduction.
For each vertex $x$ in $G_0-\hat W$ that has three neighbors in $\hat W$,
we first delete all incident edges of $x$, except for the edges with endpoints in $\hat W$. We select two neighbors of $x$ in
$\hat W$ that are independent, and delete one more edge whose endpoint is not in these two neighbors (if $x$ has three neighbors in $\hat W$).
We then contract the two edges, together with $x$,
into one.
This is possible (for otherwise, there is a
separating triangle $T$ in $G_0$. Then any coloring of $T \cup Z$ can
be extended to a $(t-1)$-coloring of interior by Corollary \ref{cor:Th5}. Note that each vertex in $G_0-W-Z$ has at most $t-6$ neighbors in $Z$).
Let $\hat W'$ be the resulting graph. Note that $G_0-Z-\hat W'$ is planar, but may have more than one components. Let
$G_P=G'-(G_0-Z-\hat W')$.
By the minimality of $G$ (and since $G_0-\hat W$ is not empty),
$G_P$ has a $(t-1)$-coloring. Let us uncontract each $x$.
Now the coloring of $\hat W'$ also gives rise to a $(t-1)$-coloring of $\hat W$
such that each $x$ (in $G_0-\hat W$ that has three neighbors in $\hat W$) sees at most two colors in the boundary of $\hat W$.
This means that every vertex outside $\hat W$ but having a neighbor in $\hat W$
has three admissible colors (because each vertex in $G_0-\hat W-Z$ has at most $t-6$ neighbors in $Z$).
Since $G_0-Z-\hat W$ is planar and since each $x$ is in the outer face boundary of the planar graph,
each vertex in the outer face boundary of $G_0-Z-\hat W$ (which may have neighbors to $\hat W$) has
three admissible colors, and every vertex in $G_0-Z-\hat W$ that is not in the
outer face boundary of $G_0-Z-\hat W$ has five admissible colors (because it
does not have a neighbor in $\hat W$).
So by Theorem \ref{thm:Th5},
we can color $G_0-Z-\hat W$ that is consistent with the $(t-1)$-coloring of $\hat W \cup Z$ (i.e., we can extend a $(t-1)$-coloring of $\hat W \cup Z$
to a $(t-1)$-coloring of the whole graph).
So we obtain a $(t-1)$-coloring of the whole graph $G$, a contradiction.

\bigskip

So far, we can find a reduction, i.e., what we have shown is the following:
\begin{quote}
If $G_0-\hat W-Z$ has at least $\alpha^2 2^{\alpha}$ vertices, then we can perform the reduction, i.e, we obtain the graph $\hat W'$ from $\hat W$ as above, and
then we can delete all the vertices in $G_0-\hat W'-Z$, because any $(t-1)$-coloring
of $G-(G_0-\hat W'-Z)$ can be extended to a $(t-1)$-coloring of the whole graph $G$. This shows that $G$ is no longer a minimal counterexample to the odd Hadwiger's conjecture for the case $t$.
\end{quote}

It remains to consider the case when $G_0-W-\hat Z$  does not have $\alpha^2 2^{\alpha}$ vertices. We will discuss this case for both Cases 2.1 and 2.2 simultaneously. By Lemma \ref{cru},
\begin{quote}
tree-width of $G_0$ is less than $10000 \alpha^{5} 2^{2\alpha} \times d(3t,g) \log n$, where $d(,)$ comes from Theorem \ref{gm7}.
\end{quote}

We will show that under this situation, Theorem \ref{main1} is satisfied.

We now start constructing a tree-decomposition $(T',Y')$ of $Y_t$ with tree-with at most $10000 \alpha^6 2^{2\alpha} \times d(3t,g) \log n$ such that
\begin{enumerate}
\item[(a)]
all children of $Y_t$ is attached to a single bag $Y'_t \in Y'$, and
\item[(b)]
the intersection between the parent bag $Y_{t'}$ of $Y_t$ and $Y_t$ is also contained in one single bag $Y'_{t'} \in Y'$
(for some $t' \in T'$).
\end{enumerate}

If we obtain such a tree-decomposition, we obtain the structure as in Theorem \ref{main1}.

Let us start a tree-decomposition  $(T'',Y'')$ of $G''_0$ which is obtained from $G'_0$ by adding edges between adjacent society vertices $w_j^i,w_{j+1}^i$ of a large vortex $(G_i,\Omega_i)$.
So $w_j^i,w_{j+1}^i \in \Omega_i$ and they are consecutive in the cyclic order of $\Omega_i$.
It is straightforward to see that tree-width of $G''_0$ is still
at most $100000 \alpha^5 2^{\alpha} \times d(3t,g) \log n$ (indeed, when we construct $W$ as above, we can start this resulting cuff. Then the rest of the arguments is the exactly same).

So, we can find such a tree-decomposition $(T'',Y'')$ of $G''_0$ in polynomial time by
Theorem \ref{tree1} because $g, d(3t,g)$ are constants.

Let us construct a tree-decomposition $(T',Y')$ of $Y_t$ from
the tree-decomposition $(T'',Y'')$ of $G''_0$ .
We first add the large vortices $\V$. This only increases the width by the factor $\alpha$ because we only need to expand each vertex of $Y''_t$ to the subgraph of order $\alpha$.
Then add $Z$ to each bag of $Y''_t$. This only increases the width by the additive factor $\alpha$.

So tree-width of the resulting decomposition is at most
$100000 \alpha^6 2^{2\alpha} \times d(3t,g) \log n$.
Then  (a) is satisfied because of 2 in Theorem \ref{thm:extended}.
Also (b) is satisfied because $Z$ is contained in every bag of $Y'_t \in Y'$. This completes the proof of Theorem \ref{main1}.\qed


\begin{thebibliography}{99}
{\footnotesize








\bibitem{4ct1}
K.~Appel and W.~Haken, Every planar map is four
colorable. {I}.
  {D}ischarging, {\it Illinois Journal of Mathematics}, {\bf 21} (1977), 429--490.

\bibitem{4ct2}
K.~Appel, W.~Haken, and J.~Koch, Every planar map is four
colorable.
  {II}. {R}educibility, {\it Illinois Journal of Mathematics}, {\bf 21} (1977),
  491--567.




\bibitem{Arn} S. Arnborg and A. Proskurowski,
Linear time algorithms for NP-hard problems restricted to
partial $k$-trees, {\em Discrete Appl. Math.} {\bf 23} (1989),
11--24.





\bibitem{bod} H. L. Bodlaender,
A linear-time algorithm for finding tree-decomposition of small
treewidth, {\it SIAM J. Comput.}, {\bf 25} (1996), 1305--1317.

\bibitem{BKMM} T. B\"ohme, K. Kawarabayashi, J. Maharry, B. Mohar,
Linear connectivity forces large complete bipartite minors,
{\it J. Combin. Theory Ser. B}, {\bf 99} (2009), 557--582.


\bibitem{catlin} P. A. \  Catlin,
A bound on the chromatic number of a graph, {\it Discrete Math} {\bf 22},
(1978), 81--83.




\bibitem{decomposition} E.~D. Demaine, M.~Hajiaghayi and K. Kawarabayashi,
Algorithmic graph minor theory: Decomposition, approximation,
and coloring, {\em 46th Annual Sumposium on Foundations of Computer
Science (FOCS 2005)}, (2005) 637--646.

\bibitem{DKM} E.~D. Demaine, M.~Hajiaghayi and K. Kawarabayashi,
Decomposition, approximation, and coloring of odd-minor-free graphs,
{\em ACM-SIAM Symposium on Discrete Algorithms (SODA'10)}, 329--344.

\bibitem{5list1} M. DeVos, K. Kawarabayashi and B. Mohar,
Locally planar graphs are 5-choosable, {\it J. Combin.\
Theory Ser.\ B}, {\bf 98} (2008), 1215--1232.



\bibitem{diestel} R.~Diestel,
{\em Graph Theory}, 2nd Edition, Springer, 2000.







\bibitem{DKMW}  R. Diestel, K. Kawarabayashi, T. M\"uller, and P. Wollan,
On the structure theorem of the excluded minor theorem of large
tree-width, {\it J.
Combin.\ Theory Ser.\ B\/}, {\bf 102} (2012), 1189--1210.




\bibitem{dirac1} G. A. Dirac,
A property of $4$-chromatic graphs and some remarks on critical
graphs, {\it J. London Math. Soc.}, {\bf 27} (1952), 85--92.


\bibitem{5list2} Z. Dvo\v{r}\'ak and K. Kawarabayashi,
List-coloring embedded graphs, {\it ACM-SIAM Symposium on Discrete Algorithms(SODA'13)}, 1004--1012.


\bibitem{epp1} D. Eppstein,
Diameter and treewidth in minor-closed graph families,
{\it Algorithmica}, {\bf 27} (2000), 275--291.


\bibitem{geelen} J. Geelen, B. Gerards, B. Reed, P. Seymour
and A. Vetta, On the odd variant of Hadwiger's conjecture, {\em J.
Combin. Theory Ser. B}, {\bf 99} (2009), 20--29.

\bibitem{ge} B. Gerards,
An extension of Konig's theorem to graphs with no odd-$K_4$.
{\em J. of Combin. Theory, Ser. B} {\bf 47} (1989), 330--348.




\bibitem{newgm} M. Grohe, K. Kawarabayashi and B. Reed,
A simpler algorithm for the graph minor decomposition  -- Logic meets Structural graph theory --,
ACM-SIAM Symposium on Discrete Algorithms, (SODA'13), 414--431.

\bibitem{guenin1} B. Guenin,
A characterization of weakly bipartite graphs, in
{\em Integer Programming and Combinatorial optimizaition},
Proceedings, 6 th IPCO conference, Houston,
Texas, 1998, (R.E. Bixby et al., Eds), Lecture Notes in Computer Science,
Vol 1412, 9--22. Springer-Verlag, Berlin, 1998.

\bibitem{guenin2} B. Guenin,
A characterization of weakly bipartite graphs,
{\em J. Combin. Theory Ser. B} {\bf 83} (2001) 112--168.



\bibitem{guenin} B. \  Guenin,
Talk at Oberwolfach on Graph Theory, Jan. 2005.




\bibitem{hadwiger} H. Hadwiger,
{\"Uber eine Klassifikation der Streckenkomplexe, {\it
Vierteljahrsschr. naturforsch.\ Ges. Z\"urich} {\bf 88} (1943),
133--142.}

\bibitem{Jansen} K. Jansen and P. Scheffler,
Generalized coloring of tree-like graphs, {\it Discerte Applied Math.}, {\bf 75} (1997), 119--130.

\bibitem{jt} T. R. Jensen and B. Toft,
Graph Coloring Problems, Wiley-Interscience, 1995.




\bibitem{kkk} K. Kawarabayashi,
Note on coloring graphs with no odd-$K_k$-minors, {\em J. Combin.
Theory Ser. B}, {\bf 99} (2009), 738--741.



\bibitem{kt} K. Kawarabayashi and B. Toft,  Any $7$-chromatic
graph has $K_7$ or $K_{4,4}$ as a minor, {\it Combinatorica}, {\bf
25} (2005), 327--353.


\bibitem{kz1} K. Kawarabayashi and Z. Song,
Some remarks on the odd Hadwiger's conjecture,
{\em Combinatorica}, {\bf 27} (2007), 429--438.

\bibitem{km} K. Kawarabayashi and B. Mohar,
Approximating the chromatic number and the list-chromatic
number of minor-closed family of graphs and odd-minor-closed family
of graphs, In {\it Proceedings of the 38th ACM Symposium on Theory
of Computing (STOC'06)}, 401--416.

\bibitem{krhadwiger} K. Kawarabayashi and B. Reed,
Hadwiger's Conjecture is decidable,
{\em the 41st ACM Symposium on Theory of Computing (STOC'09)}, 445--454.




\bibitem{short} K. Kawarabayashi and P. Wollan,
A simpler algorithm and shorter proof for the graph minor decomposition,
{\it the 43rd ACM Symposium on Theory of Computing (STOC'11)}, 451--458.

\bibitem{kkr} K. Kawarabayashi, Y. Kobayashi and B. Reed,
The disjoint paths problem in quadratic time,
{\it J. Combin. Theory Ser. B} {\bf 102} (2012), 424--435.

\bibitem{parity} K. Kawarabayashi, B. Reed and P. Wollan,
The graph minor algorithm with parity conditions,
{\it Proceedings of the 52nd Annual IEEE Symposium on Foundations of Computer Science\,(FOCS'11)}, 27--36, 2011.

\bibitem{kdh} K. Kawarabayashi, E. Demaine and M. Hajiaghayi,
Additive approximation algorithms for list-coloring minor-closed class of graphs
{\it ACM-SIAM Symposium on Discrete Algorithms, (SODA'09)}, 1166--1175.


\bibitem{kost1} A. Kostochka,
Lower bound of the Hadwiger number of graphs by their average degree,
{\it Combinatorica}, {\bf 4} (1984), 307--316.




\bibitem{mader} W.~Mader,
Existenz $n$-fach zusammenh\"angender Teilgraphen in Graphen
gen\"ugend grosser Kantendichte,
{\it Abh. Math. Sem. Univ. Hamburg} {\bf 37} (1972), 86--97.

\bibitem{MT} B. Mohar and C. Thomassen,
Graphs on Surfaces, {\em Johns Hopkins Univ.\ Press},
Baltimore, MD, 2001.


\bibitem{tree} L. P\'erkovi\'c and B. Reed,
An improved algorithm for finding tree decompositions of small
width; {\it International Journal on the Foundations of Computing
Science.}, {\bf 11}, (2000), 81--85.

\bibitem{rtree} B. Reed,
Finding approximate separators and computing tree width quickly,
In {\it
the 24th ACM Symposium on Theory of Computing(STOC'92)}.

\bibitem{4ct3} N. Robertson, D. P. Sanders, P. D. Seymour and R. Thomas,
{The four-color theorem, {\it J. Combin. Theory Ser. B} {\bf 70} (1997),
2--44.}



\bibitem{RS7} N. Robertson and P.D. Seymour,
Graph minors VII. Disjoint paths on a surface, {\it J.~Combin.
Theory Ser. B\/} {\bf 45} (1988) 212--254.

\bibitem{RS9} N. Robertson and P.D. Seymour,
Graph minors. IX: Disjoint crossed paths,
{\it J. Combin. Theory Ser. B} {\bf 49} (1990), 40--77.


\bibitem{RS10}
     N.~Robertson and P.D.~Seymour,
    Graph minors. X: Obstructions to tree-Decomposition,
    {\it J. Combin. Theory Ser. B}, {\bf 52} (1991), 153--190.

\bibitem{RS11}
   N.~Robertson and P.D.~Seymour,
   Graph minors. XI: Circuits on a Surface,
   {\it J. Combin. Theory Ser. B}, {\bf 60} (1994), 72--106.

\bibitem{RS12} N.~Robertson and P.~D.~Seymour,
Graph minors. XII. Distance on a surface, {\it J. Combin. Theory
Ser. B}, {\bf 64} (1995), 240--272.


\bibitem{RS13} N. Robertson and P.D. Seymour,
Graph minors XIII. The disjoint paths problems.
{\it J.~Combin. Theory Ser.~B}, {\bf 63} (1995), 65--110.




\bibitem{RS16}
    N.~Robertson and P.D.~Seymour,
    Graph minors. XVI: Excluding a non-planar graph,
    {\it J. Combin. Theory Ser. B}, {\bf 89} (2003), 43--76.

\bibitem{RS20} N.~Robertson and P.~D.~Seymour,
Graph minors. XX. Wagner's conjecture,
{\it J.~Combin.\ Theory Ser.~B\/}, {\bf 92} (2004), 325--357.

\bibitem{RS23} N. Robertson and P. D. Seymour,
Graph Minors XXIII, Nash-Williams' immersion conjecture,
{\it J. Combin. Theory, Ser. B}, {\bf 100} (2010), 181--205.

\bibitem{RS} N.~Robertson and P.~D.~Seymour,
An outline of a disjoint paths algorithm,
in: {\it ``Paths, Flows, and VLSI-Layout,''}
B.~Korte, L.~Lov\'asz, H.~J.~Pr\"omel, and A.~Schrijver (Eds.),
Springer-Verlag, Berlin, 1990, pp.~267--292.

\bibitem{RST1} N. Robertson, P. D. Seymour and R. Thomas,
{Hadwiger's conjecture for $K_6$-free graphs, {\it Combinatorica}, {\bf 13}
(1993), 279--361.}



\bibitem{lex} A.~Schrijver:
Combinatorial Optimization: Polyhedra and Efficiency, number 24 in
Algorithm and Combinatorics, Springer Verlag, 2003.





\bibitem{thomason1} A. ~Thomason,
An extremal function for contractions of graphs,
{\it Math. Proc. Cambridge Philos. Soc.}, {\bf 95} (1984), 261--265.








\bibitem{Th1} C.~Thomassen,
Every planar graph is 5-choosable, {\it J.~Combin. Theory Ser.~B}, {\bf 62}
(1994), 180--181.

\bibitem{carsten} C. Thomassen,
A simpler proof of the excluded minor theorem for higher surfaces.
{\it J.~Combin. Theory Ser.~B\/} {\bf 70} (1997), 306--311.












 \bibitem{Wa37} K.~Wagner,
 \"Uber eine Eigenschaft der ebenen Komplexe,
{\it Math.\ Ann.}, {\bf 114} (1937), 570--590.





}

\end{thebibliography}
\end{document}